\RequirePackage{amsmath}
\documentclass[envcountsame,envcountsect]{svmult}
\usepackage{type1cm}        % activate if the above 3 fonts are
\usepackage{makeidx}         % allows index generation
\usepackage{graphicx}        % standard LaTeX graphics tool
\usepackage{multicol}        % used for the two-column index
\usepackage[bottom]{footmisc}% places footnotes at page bottom

\usepackage{newtxtext}       % 
\usepackage[varvw]{newtxmath}       % selects Times Roman as basic font
\usepackage{caption}
\usepackage{pgfplots}
\pgfplotsset{compat=1.17}
\makeindex             % used for the subject index
\usepackage{mathtools}
\usepackage{tabularx}
\usepackage{enumitem}
\usepackage{graphicx}
\usepackage{forest}
\usepackage{booktabs}
\usepackage{url}
\renewcommand{\Gamma}{\varGamma}
\renewcommand{\epsilon}{\varepsilon}

\DeclareMathOperator{\Cir}{Cir}
\renewcommand{\leq}{\leqslant}
\renewcommand{\geq}{\geqslant}
\newcolumntype{C}{>{\centering\arraybackslash}X}
\newcommand{\R}{\mathbb{R}}
\newcommand{\Z}{\mathbb{Z}}
\newcommand{\C}{\mathbb{C}}
\newcommand{\X}{\mathbb{X}}

\newcommand{\vt}{\mathrm{vert}}
\newcommand{\conv}{\mathrm{conv}}
\newcommand{\Fix}{\mathrm{Fix}}
\newcommand{\Span}{\mathrm{Span}}
\newcommand{\core}{\mathrm{core}}
\newcommand{\PCir }{\mathrm{PCir }}

\usepackage{todonotes}
\newcommand{\cequation}[1]{%
  \multispan{2}%
  \hfill$\displaystyle{#1}$\hfill
  \ignorespaces
}
\newcommand{\lequation}[1]{%
  \multispan{2}%
  $\displaystyle{#1}$\hfill
  \ignorespaces
}
\allowdisplaybreaks

\usepackage{algorithm}
\usepackage{varioref}
\usepackage{algpseudocodex}
\usepackage{xpatch}
\makeatletter
\xpatchcmd{\algorithmic}{\itemsep\z@}{\itemsep=0.5ex plus2pt}{}{}
\makeatother

\newcommand{\OPT}{\operatorname{OPT}_{\infty}}

\begin{document}

\title*{Outer approximations of core points for integer programming}
%
%\titlerunning{Outer approximations of core points}
\author{Naghmeh Shahverdi \and
  Seyyedemahsa Banihashemi
  \and\\
  David Bremner \orcidID{0000-0003-0272-585X}}
\institute{Naghmeh Shahverdi (\emph{primary author}) \at Psychiatry and Behavioral Sciences, Stanford University, \email{naghmehshahverdi2@gmail.com} \and
  Seyyedemahsa Banihashemi \at Faculty of Computer Science, University of New Brunswick, \email{masha.banihashemi@unb.ca} \and
  David Bremner, Faculty of Computer Science, University of New Brunswick, \email{bremner@unb.ca}
}

%
%\authorrunning{ David Bremner and Naghmeh Shahverdi }
%\institute{University of New Brunswick, Fredericton, NB, E3B 5A3 , Canada \\
\maketitle  
%
           % typeset the header of the contribution
%
\abstract{
For several decades the dominant techniques for integer linear programming have been branching and cutting planes. Recently, several authors have developed core point methods for solving symmetric  integer linear programs (ILPs). An integer point is called a core point  if its orbit polytope is lattice-free. It has been shown that for symmetric ILPs, optimizing over the set of core points gives the same answer as considering the entire space. Existing core point techniques rely on the number of core points (or equivalence classes) being finite, which requires special symmetry groups.  In this paper we develop some new methods for solving symmetric ILPs --- based on outer approximations of core points --- that do not depend on finiteness but are more efficient if the group has large disjoint cycles in its set of generators.
}
\section{Introduction}

Formulation symmetries occur in practice when relabellings yield equivalent problem structure; this causes repeated work for branching solvers, and state of the art commercial and research solvers make efforts to break symmetries~\cite{PR2019}. 
Let $ G\leq S_{n} $ be a permutation group acting on $ \R^{n} $ by permuting coordinates. For any integer point $ z\in \Z^{n} $,  the \emph{orbit polytope} of $z$ is the convex hull of the $G$-orbit of $z$.  If the vertices of  an orbit  polytope are the only integer points in the polytope  we call it \emph{lattice-free} and call $z$ a \emph{core point}.  Instead of seeing symmetry as a problem, core point techniques seek to exploit it to solve  integer linear programs (ILPs) faster. In the most direct approach, when the number of core points  is finite (which only holds for certain special groups),  core points are enumerated and tested individually~\cite{BHJ2013}.  It should be noted that core point techniques are not useful for binary problems since every $\{0,1\}$-point is a core point; \cite{HP2019} considers an alternative approach based on lexicographical order.

Computation of symmetries  in the MIPLIB 2010 and 2017 instances has been done in \cite{miplib} and this study shows that  many instances are affected by symmetry. Symmetric Integer Linear Programming appears in many problems such as scheduling on identical machines and code construction. For solving these problems, Artificial Intelligence approaches have been investigated in \cite{AP1,AP2}. 

A (not necessarily polyhedral)  outer approximation is a set of constraints that is feasible for all of the points in the set one wishes to approximate. A well known example of an outer approximation is an ILP, where the (initial) linear constraints define an outer approximation of the feasible integer points. Outer approximations lead naturally to a hybrid approach where synthesized constraints are added to an existing formulation and then solved with a traditional solver. Outer approximations are implicit in previous results bounding the distance of core points to certain linear subspaces (see e.g.\ Theorem 3.24 in  \cite{2}). The distance bounds do not themselves seem to be tight enough to provide a practical improvement for solving ILPs. In this paper we develop some new constraints for problems with  formulation symmetries. While these constraints are nonlinear and non-convex, initial experiments with nonlinear solvers seem promising.

In  Section 2 we give some basic definitions. In  Section  3 we consider integer linear programs with  cyclic symmetry groups and provide some new constraints to determine  outer approximations of their core points. We also provide an algorithm that uses these constraints to solve an ILP.  In Section 4 we  generalize this algorithm for ILPs where only  some of their  variables have cyclic symmetry. In Section 5 we generalize the algorithm for direct products of cyclic groups. In Section 6  we classify permutation groups based on their generators and explain  how the algorithms of the previous sections can be applied to ILPs
having arbitrary permutation groups as symmetry groups.  Finally in Section 7 we use our methods to solve  some hard symmetric integer linear feasibility problems.  

\section{Basic Definitions}

Symmetries of geometric objects (e.g. polyhedra and integer lattices) in
integer linear programming can be viewed as the action of some
underlying group.

\begin{definition}
If $ G $ is a group and $ X $ is a set, then a (left) \emph{group action} $  \phi$ of $ G $ on $ X $ is a function 
\begin{align*}
\phi : G \times X & \rightarrow X ,\\
 (g,x) & \rightarrow gx,
\end{align*}
that satisfies the following two axioms.
\begin{enumerate}
\item ex = x for all $ x\in X $, where  $e \in G$ is the identity element.
\item  \emph{(}gh\emph{)}x = g\emph{(}hx\emph{)} for all $g, h \in G$ and all $x \in X$. 
    \end{enumerate}
\end{definition}

In this paper we assume that all groups considered are permutation
groups, and that they act on a family $\X^n$ of $n$-tuples (e.g.\
$\C^n$, $\R^n$, or $\Z^n$ in the usual coordinates) as follows.
\begin{remark}\label{2}
The permutation group $ G \leq S_{n} $ acts on $ \X^{n} $ by  permuting coordinates $\{\,0, \ldots n-1\,\}$: for $ g \in G $ and $ x=(x_{0},\ldots,x_{n-1}) \in \X^{n} $ 
\begin{equation*}
\phi_{g}(x)=(x_{g^{-1}(0)},x_{g^{-1}(1)},\ldots,x_{g^{-1}(n-1)}).
\end{equation*}
\end{remark}
There are certainly more general notions of symmetry possible, but
permuting coordinates is the most widely studied class of symmetries in integer linear
programming~\cite{PR2019}.

The techniques in this paper rely on restricting the search for
feasible integral points to certain special integral points called
\emph{core points}. Core points are defined relative to a given group
as follows.

\begin{definition}[Core Points]
  \leavevmode
  \begin{enumerate}
  \item Let $ P \subset \mathbb{R}^n $ be a convex polytope with integral vertices. We call $ P $ \emph{lattice-free}
if $ P \cap \mathbb{Z}^n= \vt (P) $ where $ \vt (P) $ is the set of vertices of~$P$.
\item Let $ G \leq GL_{n}(\mathbb{R})  $ be a finite group and let $ G_{z} $ be
the G-orbit of some point $ z \in \mathbb{\R}^n$. We call the convex hull of this orbit an \emph{orbit polytope} and denote it by $\conv(G_{z})$.
\item Let $ G \leq GL_{n}(\mathbb{Z}) $ be a finite group of unimodular matrices. A point
$ z \in \mathbb{Z}^n $ is called a \emph{core point} for $ G $ if and only if the orbit polytope $\conv(G_{z})$ is lattice-free.
  \end{enumerate}
\end{definition}
 
Let $ C_{n}=\langle \sigma \rangle $ denote the \emph{cyclic group}
generated by $ \sigma $, a cyclic permutation of coordinates. In other
words, $ \sigma: \mathbb{R}^n \rightarrow \mathbb{R}^n $ is given by
\begin{equation*}
  \sigma(c_{0},c_{1},\ldots,c_{n-1})=(c_{n-1},c_{0},\ldots,c_{n-2})
\end{equation*}
We represent points in $ \mathbb{R}^n $ by column vectors but for convenience, we write such vectors here in a transpose way.  The map $ \sigma $ is easily iterated:
\begin{equation*}
\sigma^{k}(c_{0},c_{1},\ldots,c_{n-1})=(c_{n-k},c_{n-k+1},\ldots,c_{n-1},c_{0},\ldots,c_{n-k-1})
\end{equation*}

Consider a cyclic group $ C_{4}=\langle (1,2,3,4)\rangle $. The orbit
of $ x=(2,0,4,1) $ consists of the four vectors $ x_{1}=(2,0,4,1)$,
$ x_{2}=(1,2,0,4)$, $ x_{3}=(4,1,2,0)$, $ x_{4}=(0,4,1,2)$. The orbit
polytope of $ x $ is not lattice-free because
 \begin{equation*}
   (1,2,1,3)= \frac{1}{5}x_{1}+\frac{3}{5}x_{2}+0x_{3}+\frac{1}{5}x_{4},
\end{equation*}
 so that $ x $ is not a core point.

\begin{definition}
Let $ G $ act on $\C^n$. The subset of $\C^n$ preserved by all elements of $G$ is called the \emph{fixed space}:
\begin{equation*}
\Fix(G) :=\{x\in \C^n \; \;  \vert  \; \;   gx=x, \forall g\in G\}.
\end{equation*}
We denote  by $ \Fix_{\Z}(G) $ (resp.\ $\Fix_{\R}(G)$), the intersection $ \Fix(G) \cap \Z^{n}$ (resp.\ $ \Fix(G) \cap \R^{n}$).
\end{definition}

We use $\mathbf{1}^k$ to denote the $k$-dimensional vector of all ones
(or just $\mathbf{1}$ where the dimension is clear from context).  It
is easy to see that $\mathbf{1}^n$ is contained in $\Fix_{\Z}(G)$ for  any group $G$ acting on $\R^n$ by permuting coordinates. If $G$ is
further transitive, then $\Fix_{\R}(G) = \Span(\mathbf{1})$.

\begin{definition}
We define the \emph{$k$-th layer} to be the set
\begin{equation*}
  \Z_{(k)}^{n}:=  \lbrace  z\in \mathbb{Z}^n \; \; \vert \; \;  \langle z,\mathbf{1}\rangle=k \rbrace.
\end{equation*}
\end{definition}

Note that the set $ \Z_{(k)}^{n} $ is $G$-invariant because $ G $ acts by permuting coordinates.
Consider the cyclic group $C_{4}= \langle (1,2,3,4)\rangle $. Let $X=\Z_{(2)}^4 \cap \{0,1\}^4$ be the  set of $\{0,1\}$-points in $ \R^{4} $ in layer two, i.e., each point has two $1$s and two $0$s, and suppose the action is the same as in Remark~\ref{2}. This action has two orbits:
\begin{align*}
 O_{G}((1,1,0,0)) &=\lbrace\, (1,1,0,0),(0,1,1,0),(0,0,1,1),(1,0,0,1)\,\rbrace\,, \\
O_{G}((1,0,1,0)) &=\lbrace\, (1,0,1,0),(0,1,0,1)\, \rbrace\,.
\end{align*}
The fixed space $\Fix_{\Z}(C_{4})$ contains only two $0,1$ points, namely $ (1,1,1,1) $ and $(0,0,0,0)$, and neither is in $X$.

In the remainder of the paper we will need several different notions of equivalence for integral points.
Potentially larger equivalence classes (based on normalizers) of core points are studied in \cite{LS2018}.
\begin{definition}[Equivalence relations]
  \label{equiv}
\leavevmode
\begin{enumerate}
\item Two points $ x,y \in \mathbb{Z}^n $ are called \emph{equivalent}  if there exists  $ g \in G $ such that $ x=gy $. It follows from the group axioms that this is an  equivalence relation.
\item Two points $ x,y \in \Z^{n} $ are called \emph{isomorphic} if there exists $ g \in G $ such that $ x-gy \in \Fix_{\Z}(G) $. This is an equivalence relation because $\Fix_{\Z}(G)$ is a lattice.
\item Two integer points $ z_{1} $ and $ z_{2} $ in $ \Z^{n} $ are called \emph{co-projective} if there exists an integer $ k\in \Z $ such that $ z_{1}=z_{2}+k\mathbf{1}$. Equivalently, if the group is transitive, $z_2$ is a translation of $ z_{1} $ through the fixed space.
\end{enumerate}
\end{definition}

Each of the equivalence relations in Definition~\ref{equiv} has
equivalence classes that are either entirely core-points or entirely
non-core integer points. Along with the observation that
$\{0,1\}$-vectors are core points for any permutation
group~\cite[Lemma 3.7]{2}, we can define a family of \emph{universal}
core points.

\begin{definition}
A point $ u \in \mathbb{Z}^n $ is called a \emph{universal core point} if it is isomorphic to a $\{0 , 1\}$-vector.
\end{definition}

As a heuristic for identifying more points contained in the orbit
polytopes of many (non-core) points, we consider integer points near
universal core points.

\begin{definition}
An integer point $ z $ is called an \emph{atom} if there is a universal core point $ u $ in the layer containing $ z $ such that the distance between $ z $ and $ u $ is~$\sqrt{2}$.
\end{definition}

For example if $ G=\langle (1,2,3,4,5)\rangle $ then the fixed space
is spanned by $\mathbf{1}$ and $ (2,2,2,2,1)=(1,1,1,1,0)+(1,1,1,1,1) $
is a universal core point. The point $ (3,2,2,1,1) $ is an atom since
$ (3,2,2,1,1)-(2,2,2,2,1)=(1,0,0,-1,0)$ hence its distance to the
universal core point $ (2,2,2,2,1) $ is $ \sqrt{2} $.

\begin{definition} \label{active}
Suppose the cyclic group $ G=\langle (g_{1},\ldots,g_{k})\rangle \leq S_{n}$, $ g_{i}\in \{0,\ldots,n-1\} $ acts on $\R^n$  by permuting coordinates.  Then  coordinate $i$ is called \emph{active} if there is $ j\in \{1,\ldots,k\} $ such that $ g_{j}=i$.
 \end{definition}
For example if $ G=\langle (0,1,3,4)\rangle $ acts on $ \R^{5} $ then $ x_{0},x_{1},x_{3},x_{4} $ are active but $ x_{2}  $ is non-active.
\section{Circulant Matrices}
\label{sec:circulant}
Circulant matrices play an important role in finding our constraints because any point $ x $ in the orbit polytope of the integer point $ c\in \Z^{n} $ under the cyclic group $ C_{n} $ can be written as $x=C\lambda$, where $ x \in \R^{n} $, $ \lambda \in [0,1]^{n} $ and $C$ is the  \emph{circulant matrix} of $ c $. 

\begin{definition} A \emph{circulant matrix} is a matrix where each column vector is rotated one element down relative to the preceding column vector. An $ n\times n$  circulant matrix  $\Cir(c)$ takes the form
\begin{equation*}
 C=\begin{bmatrix}
c_{0} & c_{n-1} & \ldots & c_{2} & c_{1}\\
c_{1} & c_{0} & c_{n-1} &\ldots  & c_{2}\\
\vdots & c_{1} & c_{0} & \ddots & \vdots\\
c_{n-2}& & \ddots &\ddots & c_{n-1}\\
c_{n-1}& c_{n-2}& \cdots & c_{1} & c_{0}
\end{bmatrix}. 
\end{equation*}
\end{definition}

One amazing property of circulant matrices is that the eigenvectors are always the same for all $ n\times n $ circulant matrices.  The eigenvalues are different for each matrix, but since we know the eigenvectors a circulant matrix can be diagonalized easily. For more detailed background on circulant matrices see~\cite{Cir}.

The $ m$-th eigenvector $ y^{m} $  for any $n\times n$ circulant matrix $ \Cir(c) $ is given by:
\begin{equation}\label{9}
y^{m}=\frac{1}{\sqrt{n}}(1,w^{-m}_{n},\ldots,w_{n}^{-(n-1)m})^{T},
\end{equation} 
where  $ w^{m}_{n}=\exp(2\pi m i/{n}). $ Suppose $ c=(c_{0},\ldots{},c_{n-1})\in \R^{n} $ (usually for us $ c $ will be an integer point in $ \Z^{n} $). By Euler's formula we have $ \sqrt{n}y^{m} =V_{m}-iU_{m},\; \; m=0,\ldots{},n-1 $, where
\begin{equation}\label{cos}
V_{m} =\left( 1, \cos\left( \frac{2\pi m}{n}\right) ,\ldots{},\cos\left( \frac{2\pi (n-1)m}{n}\right) \right), 
\end{equation}
\begin{equation}\label{sin}
U_{m} =\left( 0, \sin\left( \frac{2\pi m}{n}\right) ,\ldots{},\sin\left( \frac{2\pi (n-1)m}{n}\right) \right). 
\end{equation}

The eigenvalue of the $m$-th eigenvector of the circulant matrix $ \Cir(c) $ is
\begin{align*}
\psi_{m}&=\sum_{k=0}^{n-1}c_{k}w^{km}_{n}=\langle V_{m},c \rangle +i\langle U_{m},c\rangle. \\
\shortintertext{So we have}
\Cir(c)Y&=Y\Psi \; \; \; \; \Longrightarrow \; \; \; \; \; \Cir(c)=Y\Psi Y^{*}, \\
\shortintertext{where}
Y &=[y^{0}\mid \ldots\mid y^{n-1}],
\end{align*}
is the unitary matrix composed of the eigenvectors as columns, and $ \Psi $ is the diagonal matrix with diagonal elements $ \psi_{0},\ldots,\psi_{n-1} $.

The inverse of a circulant  matrix is circulant \cite{Book} and its inverse is given by
\begin{equation}\label{inverse}
\Cir(c)^{-1}=Y\Psi^{-1}Y^{*}.
\end{equation}
Since $ \Psi $ is a diagonal matrix its inverse is also a diagonal matrix with diagonal elements $ \psi^{-1}_{m} $, $m=0,\ldots,n-1$ where
\begin{equation}\label{length}
\dfrac{1}{\psi_{m}}=\dfrac{1}{\sum_{k=0}^{n-1}c_{k}w^{km}_{n}}=\dfrac{1}{\langle V_{m},c \rangle+i\langle U_{m},c\rangle}
=\dfrac{\langle V_{m},c \rangle-i\langle U_{m},c\rangle}{\langle V_{m},c \rangle^{2}+\langle U_{m},c\rangle^{2}}. 
\end{equation}

\begin{remark}
Note that the length of the projection of a vector $ c\in \R^{n} $ onto a complex vector $ v = a+ib\in \C^{n} $ is defined as
\begin{equation}\label{3}
\Vert \mathrm{Proj}^{c}_{v} \Vert^{2}= \dfrac{\Vert \langle c, a\rangle-i\langle c,b \rangle \Vert^2}{\Vert v \Vert^{2}} = \dfrac{\langle c, a\rangle^{2}+\langle c,b \rangle^{2}}{\Vert v \Vert^{2}}.
\end{equation}
Furthermore,  the term $ \langle V_{m},c \rangle^{2}+\langle U_{m},c\rangle^{2} $ in~\eqref{length} is the length of the projection of $c$ onto invariant subspace $ y^{m} $.
\end{remark}

\begin{lemma}\label{T} Let $ c\in \R^{n} $ and suppose $\Cir(c)$ is invertible. Then its inverse is $\Cir( \widehat{T}(c))$,
where $ \widehat{T}(c) $ is defined as follows:
\begin{equation*}
  \widehat{T}(c) =
\setlength{\arraycolsep}{2.5pt}
\medmuskip = 0mu  \dfrac{1}{n}\begin{bmatrix} \dfrac{1}{\langle c,\mathbf{1} \rangle}+ T_{0}(c)\\\dfrac{1}{\langle c,\mathbf{1} \rangle}+ T_{1}(c)\\ \vdots \\\dfrac{1}{\langle c,\mathbf{1} \rangle} +T_{n-1}(c) \end{bmatrix}=\dfrac{1}{n}\begin{bmatrix}
           \dfrac{1}{\langle c,\mathbf{1} \rangle}+\psi^{-1}_{1}+\ldots+\psi^{-1}_{n-1}\\
\dfrac{1}{\langle c,\mathbf{1} \rangle}+w_{n}^{-1}\psi^{-1}_{1}+\ldots+w^{-(n-1)}_{n}\psi^{-1}_{n-1}\\
\vdots\\
\dfrac{1}{\langle c,\mathbf{1} \rangle}+w^{-(n-1)}_{n}\psi^{-1}_{1}+..+w^{-(n-1)^{2}}_{n}\psi^{-1}_{n-1}
\end{bmatrix}.
\end{equation*}
Note that actually $ \dfrac{1}{\langle c,\mathbf{1} \rangle} = \psi_{0}^{-1}. $

\end{lemma}
\begin{proof}
By~\eqref{inverse} we have $ \Cir(c)^{-1}=Y\Psi^{-1} Y^{*} $. Now suppose $ k_{hj} $ is the $(h,j)$-th component of $Y\Psi^{-1}$. We have
\begin{equation*}
 k_{hj}= Y_{hj} \Psi_{jj}^{-1}=\dfrac{1}{\sqrt{n}}w^{-(h-1)(j-1)}_{n}\Psi^{-1}_{(j-1)(j-1)} .
\end{equation*}

So, $k^{l}$, the $l$-th row of $Y\Psi^{-1}$,  is equal
\begin{equation*}
k^{l}=\dfrac{1}{\sqrt{n}}\left[\Psi_{00}^{-1}, w^{-(l-1)}_{n}\Psi^{-1}_{11}, w^{-2(l-1)}_{n}\Psi^{-1}_{22},\ldots,w_{n}^{-(n-1)(l-1)}\Psi^{-1}_{(n-1)(n-1)}\right] \,.
\end{equation*}
Now since $ \Cir(c)^{-1} $ is a circulant  matrix~\cite{Book} it is enough to find the first column of $ \Cir(c)^{-1} $ (the other columns can be found by permutation of the first column). Notice that the first row and column of $ Y $ and $ Y^{*} $ is $\frac{1}{\sqrt{n}}\mathbf{1}=\frac{1}{\sqrt{n}}(1,1,\ldots,1)$ so multiplying each row of $ Y\Psi^{-1}$ with vector $\frac{1}{\sqrt{n}}\mathbf{1} $ gives us  the first column of  $ Y\Psi ^{-1}Y^{*} $, which  is
\begin{equation*}
\langle k^{l},\frac{1}{\sqrt{n}}\mathbf{1} \rangle=\frac{1}{n}(\Psi^{-1}_{00}+w^{-(l-1)}_{n}\Psi^{-1}_{11}+ w^{-2(l-1)}_{n}\Psi^{-1}_{22}+\ldots+w^{-(n-1)(l-1)}_{n}\Psi^{-1}_{(n-1)(n-1)}),
\end{equation*}
where $ \Psi^{-1}_{00}={1}/{\langle c,\mathbf{1}\rangle}.$
\end{proof}

The following theorem plays an important role in this paper, and is a useful formula for computing the inverse of an invertible circulant matrix.

\begin{theorem}\label{7} For $c\in \R^n$, the entry  $ T_{k}(c) $ in the column vector $ \widehat{T}(c) $ in  \emph{Lemma~\ref{T}} can be written as
\begin{align*}
T_{k}(c)& =2\sum_{m=1}^{(n-1)/2}\dfrac{1}{\langle V_{m},c\rangle^{2}+\langle U_{m},c\rangle^{2} }\langle \sigma^{-k}(V_{m}),c\rangle  \;\;\; \;\; \text{if n is odd}, \\
T_{k}(c)&=2\sum_{m=1}^{(n-2)/{2}} \dfrac{1}{\langle V_{m},c\rangle^{2}+\langle U_{m},c\rangle^{2} }\langle \sigma^{-k}(V_{m}),c\rangle + \dfrac{(-1)^{k} }{\langle V_{\frac{n}{2}},c\rangle } \;\;\; \text{if n is even}.
\end{align*}
In particular each $ T_{k}(c) $ is a real number.
\end{theorem}

\begin{proof}  As  shown in Lemma~\ref{T} we have
  \begin{equation*}
    T_{k}(c)=w^{-k}_{n}\psi^{-1}_{1}+\ldots+w^{-(n-1)k}_{n}\psi^{-1}_{n-1}.
  \end{equation*}
Since $ \psi_{l} $ and $ \psi_{n-l} $  are complex conjugates of each other  we have
  \begin{multline*}
    w^{-lk}_{n}\psi^{-1}_{l}+w^{-(n-l)k}_{n}\psi_{n-l}^{-1}=
    ( \cos(\frac{2lk\pi }{n})-i\sin(\frac{2lk\pi }{n}))\dfrac{\langle c,V_{l}\rangle -i\langle c,U_{l}\rangle}{\langle c,V_{l}\rangle^{2} +\langle c,U_{l}\rangle^{2}}\\
    +(\cos(\frac{2lk\pi }{n})+i\sin(\frac{2lk\pi }{n}))\dfrac{\langle c,V_{l}\rangle +i\langle c,U_{l}\rangle}{\langle c,V_{l}\rangle^{2} +\langle c,U_{l}\rangle^{2}}=\\
\dfrac{2\cos(\frac{2lk\pi}{n})\langle c, V_{l} \rangle-2\sin(\frac{2lk\pi }{n})\langle c,U_{l}\rangle}{\langle c,V_{l}\rangle^{2} +\langle c,U_{l}\rangle^{2}}=
2\dfrac{\langle c,\sigma^{-k}(V_{l}) \rangle}{\langle c,V_{l}\rangle^{2} +\langle c,U_{l}\rangle^{2}}.
  \end{multline*}
Recall that $ w^{m}_{n}=\cos\left( \frac{2\pi m}{n}\right) +i\sin\left( \frac{2\pi m}{n}\right) $ so the last equality holds because $ \cos A\cos B-\sin A\sin B=\cos\left( A+B\right) $ and $ V_{m} $ is in terms of $\cos$ and $ U_{m} $ is in terms of $ \sin $  (see~\eqref{cos} and~\eqref{sin}).

If $ n $ is even then for $ l=\frac{n}{2} $ we have $ \psi_{\frac{n}{2}}=\psi_{n-l} $ and $ y^{\frac{n}{2}}=y^{n-l} $ so $w^{-\frac{n}{2}k}_{n}\psi^{-1}_{\frac{n}{2}}$ does not give a complex conjugate pair. But since the imaginary part of it is zero, we have

\begin{equation*}
 w^{-\frac{n}{2}k}_{n}\psi^{-1}_{\frac{n}{2}}= \cos\left( -k\pi\right) \dfrac{\langle c, V_{\frac{n}{2}}\rangle}{\langle c, V_{\frac{n}{2}}\rangle^{2}}=\dfrac{\cos\left( k\pi\right) }{\langle c, V_{\frac{n}{2}}\rangle}=\dfrac{(-1)^{k}}{\langle c, V_{\frac{n}{2}}\rangle}. \qedhere
\end{equation*}
\end{proof}

\begin{lemma}\label{T0}
Suppose for $ c\in \R^{n} $ that $\Cir(c)$ is invertible. Then   $\sum_{k=0}^{n-1} T_{k}(c)=0 $.
\end{lemma}
\begin{proof}
  From Lemma~\ref{T} we have
  \begin{align*}
    \sum_{k=0}^{n-1} T_{k}(c) &= \sum_{k=0}^{n-1}\sum_{m=1}^{n-1}w_n^{-km} \psi_m^{-1}= \sum_{m=1}^{n-1} \psi_m^{-1} \sum_{k=0}^{n-1}w_n^{-km}\,.\\
    \shortintertext{Since the inner sum is geometric, for $m<n$}
    \sum_{k=0}^{n-1}w_n^{-km} &= 0\,. \qedhere
  \end{align*}
\end{proof}

\begin{remark}\label{aff}
Note that if $ z,c \in \Z_{(k)}^{n}, \; k\neq 0, $ and $ z=\Cir(c)\lambda $ then 
\begin{equation*}
  k= \langle \mathbf{1},z\rangle = \langle \mathbf{1},\Cir(c)\lambda \rangle = \langle \mathbf{1},c \rangle \langle \mathbf{1},\lambda \rangle=k\langle \mathbf{1},\lambda \rangle.
 \end{equation*}
This implies that $\langle \mathbf{1},\lambda \rangle=1.$  Now suppose  $ \Cir(c) $ is invertible, so $ \lambda= \Cir(c)^{-1}z $. To check if $ z \in \conv(G_{c}) $ or not, we only need to check if $ \lambda= \Cir(c)^{-1}z \geq \mathbf{0}$. The constraint $\sum_{i=0}^{n-1}\lambda_{i}=1$ is redundant.
\end{remark}

Let us denote by $ \overline{T}(c) $ the first row of $ \Cir(\widehat{T}(c)) $ which is  
\begin{equation*}
\overline{T}(c) =  \dfrac{1}{n}\left[   \dfrac{1}{\langle c,\mathbf{1} \rangle}+T_{0}(c),\dfrac{1}{\langle c,\mathbf{1} \rangle}+T_{n-1}(c),\ldots{},\dfrac{1}{\langle c,\mathbf{1} \rangle}+T_{1}(c) \right]. 
\end{equation*}
This  will simplify the notation in following sections.
\subsection{New Constraints for Singular Circulant Matrices  }
The circulant matrix $\Cir(c)$ corresponding to an integer point $c\in \Z^{n}$ is not invertible if and only if the determinant of $ \Cir(c) $ is zero. Since the determinant of a square matrix is equal to the product of its  $ n $ eigenvalues we have
\begin{equation*}
\det(\Cir(c))=\prod_{j=0}^{n-1} (\langle V_{j},c \rangle +i\langle U_{j},c\rangle)\,.
\end{equation*}
Furthermore, since for $ k=1,\ldots{},n-1 $,  $ V_{n-k}+iU_{n-k} $ is the complex conjugate of $ V_{k}+iU_{k} $ we have 
\begin{equation*}
  \det(\Cir(c))=
  \begin{cases}
    \langle V_{0},c \rangle \prod_{j=1}^{(n-1)/{2}}(\langle V_{j},c \rangle^{2} +\langle U_{j},c\rangle^{2}) &  n \text{ odd,}\\[1.1ex]
    \langle V_{0},c \rangle \langle V_{\frac{n}{2}},c\rangle \prod_{j=1}^{(n-2)/{2}}(\langle V_{j},c \rangle^{2} +\langle U_{j},c\rangle^{2}) & n \text{ even.}
  \end{cases}
\end{equation*}
So in order to check if an ILP has an integer solution whose  circulant matrix is singular, we could add the following constraints to the problem
\begin{align}
  \label{invert}
  \langle V_{0},c \rangle \prod_{j=1}^{(n-1)/{2}}(\langle V_{j},c \rangle^{2} +\langle U_{j},c\rangle^{2})&=0 \; \;  \; \; \;\; \;  \text{if} \; n \; \text{is odd},\\
  \label{inverT1}
\langle V_{0},c \rangle \langle V_{\frac{n}{2}},c\rangle \prod_{j=1}^{(n-2)/{2}}(\langle V_{j},c \rangle^{2} +\langle U_{j},c\rangle^{2})&=0 \; \; \;\; \; \; \text{if} \; n \; \text{is even}.
\end{align}
Equations~\eqref{invert} and~\eqref{inverT1} involve polynomials in $c$ of degree $n$. We can reformulate them in a way that make them easier to solve in practice by introducing a large positive constant $M$ and ${0,1}$-valued variables
$ r_{m} $. Let  $ P_{j} $ be a term in the products~\eqref{invert} and~\eqref{inverT1}, namely,
$ P_{j}=\langle V_{j},c\rangle $ when $j = 0$ or (for $ n $ even) when
$j = \frac{n}{2}$. Otherwise,
$ P_{j}= \langle V_{j},c \rangle^{2} +\langle U_{j},c\rangle^{2} \geq0 $.
For $j=0,\ldots{}, \lceil ({n-1})/{2}\rceil $, add constraints
\begin{align}
  \label{in}
  P_{j} & \leq  r_{j}  P_{j}  &  j \not \in \{ 0,\frac{n}{2} \} \cap \Z,\\
  -  r_{j}  M  \leq P_{j} &\leq  r_{j}  M    &\text{otherwise}. \nonumber
\end{align}
Finally add the following constraint to force at least one $r_j$ to $0$:
\begin{equation}
\label{on}
\sum_{j=0}^{\left\lceil (n-1)/{2}\right\rceil} r_{j}  \leq \left\lceil \frac{n-1}{2}\right \rceil\,.
\end{equation}
Note that constraints~\eqref{in} and~\eqref{on} forces at least one of
the $P_{j}$ to be zero and so the determinant will be zero.  For
maximization problems the constant $M$ can be chosen as the absolute
value of the objective value of the LP relaxation (assuming the
objective function $f(x)=\langle \mathbf{1},x \rangle$).

\begin{remark}
We can formulate~\eqref{invert} and~\eqref{inverT1} in different ways. For example we can make $ \left\lceil \frac{n-1}{2}\right\rceil$ subproblems by adding each  $ P_{m} $ separately. The corresponding constraints in each subproblem are linear. For example in the $ h $-th subproblem, $ h\in \{1,\ldots, \left\lceil \frac{n-1}{2}\right\rceil \}$, we have 
\begin{equation*}
P_{h}= \langle V_{h},c \rangle^{2} +\langle U_{h},c\rangle^{2} =0,
\end{equation*}
which can be simplified as
\begin{equation*}
\langle V_{h},c \rangle =0 \text{ and } \langle w_{h},c \rangle =0.
\end{equation*}
The weakness of this formulation is that adding these constraints does not simplify the problem enough because in each step we are searching in $ n-2 $ dimensions which for large $ n $ is not a sufficient dimension reduction.  
\end{remark}
\subsection{New Constraints for Non-Singular Circulant Matrices  }
In this section we develop some new constraints to find an outer approximation of core points. These new constraints  depend only on the symmetry group and not the ILP.

Suppose $ c $ is an arbitrary integer point in $\Z^{n}_{(k)}, \; k\neq 0,$ and $ G $ is the cyclic group with order $ n $ which acts on coordinates as usual. By  Remark~\ref{aff} a point $ z \in \Z^{n}_{(k)}$ is in $\conv(G_{c})$ if and only if there exist $ \lambda=(\lambda_{0},\ldots,\lambda_{n-1}) $ such that
\begin{align}
  \SwapAboveDisplaySkip
  z&=\Cir(c)\lambda, \qquad \lambda_{i} \geq 0.\nonumber\\
  \shortintertext{If $ \Cir(c) $ is invertible  we have}
  \label{lambda_def}
  \lambda& = \Cir(c)^{-1}z, \nonumber\\
  \shortintertext{and by Lemma~\ref{T}}
  \lambda_{j}&= \langle (\sigma^{j}(\overline{T}(c)),z\rangle = \langle \overline{T}(c), \sigma^{-j}(z) \rangle\,.               
\end{align}

The orbit polytope of a core point is lattice-free. Still assuming that $\Cir(c)$ is invertible, we conclude that  $ c \in \mathbb{Z}^{n}_{(k)} $ is a core point if and only if for each $ z \in \mathbb{Z}^{n}_{(k)} $ which is not  in the $ G $-orbit of $  c$, 
\begin{equation}\label{lambda}
\min_j\;  \langle \overline{T}(c), \sigma^{-j}(z) \rangle  <0,
\end{equation}
Since $ c $ and $ z $ lie in the same layer $\langle c, \mathbf{1} \rangle=k  $, $\lambda_j<0$ is equivalent to (interpreting indices modulo $n$)
\begin{equation*}
 z_{n+j-0}T_{0}(c)+z_{n+j-1}T_{1}(c)+\ldots+z_{n+j-(n-1)}T_{n-1}(c)+1 <0.
\end{equation*}
Informally, we may say that for at least one permutation of subscripts, we have
\begin{equation}\label{6}
 z_{0}T_{0}(c)+z_{1}T_{1}(c)+\ldots+z_{n-1}T_{n-1}(c)+1 <0. 
\end{equation}
It should be mentioned that, for almost all $ z \in \Z^{n}$ the
constraint $ \langle z, \overline{T}(c)\rangle <0 $ is nonlinear and
non-convex with respect to $ c $.  Considering a single $z$ in
isolation, we can assume $ \lambda_{0} $ is negative, because
if for a core point $ c $, $ \lambda_{0} $ is not negative, there is a
permutation of $ c $ that has negative~$ \lambda_{0} $.

Define
\begin{equation*}
 H(z)=   z_{0}T_{0}(c)+ z_{1}T_{1}(c)+\ldots+z_{n-1}T_{n-1}(c)+1 ,
\end{equation*}
so that $ H(z)=0 $ is the equation of a hyperplane  in $ \mathbb{R}^n $ with non-zero normal vector  $ ( T_{0}(c),T_{1}(c), \ldots, T_{n-1}(c)) $ which is perpendicular to the fixed space (Lemma~\ref{T0}). (If all $T_{j}(c) = 0$, then $ \widehat{T}(c) $ in Lemma~\ref{T} is a multiple of $ \mathbf{1} $. But this makes $\Cir(c)^{-1}$ non-invertible). More precisely, $ c \in \mathbb{Z}^n_{(k)} $ is a core point if for all integer points $ z\in \mathbb{Z}^n_{(k)} $ there is an index $ j\in [n] $ such that $ H(\sigma^{j}(z))<0 $.

Suppose  $ Q $ is the orbit polytope of a non-core point and $ R $ is the orbit polytope of a universal core point in the same layer. Intuitively, we expect that  many integer points whose orbit polytope contains $ Q $, also have orbit polytope containing  $ R $. The idea for making new constraints is to remove from the feasible region all integer points whose orbit polytopes contain atoms or universal core points. This process can be done by searching layer by layer in the feasible region.

\begin{lemma}\label{iso}
For two co-projective integer points $ c $,$ c^{\prime} \in \Z^{n}$ we have 
\begin{equation*}
  T_{j}(c)=T_{j}(c^{\prime}) \text{ for all } \; j=0,\ldots,n-1 .
\end{equation*}
Moreover, the constraints~\eqref{lambda} are invariant under translation in the fixed space. That is, if $ z $ and $ z^{\prime}$ are two co-projective integer points in the same non-zero layers as $ c $ and $ c^{\prime} $ respectively, we have
\begin{equation*}
 \langle z, \overline{T}(c) \rangle= \langle z^{\prime} ,\overline{T}(c^{\prime}) \rangle.
\end{equation*}
\end{lemma}
\begin{proof}
 Recall from Definition~\ref{equiv} that $c'=c+k\mathbf{1}$ for some $k \in \mathbb{Z}$.
Since $\mathbf{1}$ is orthogonal to the other  eigenvectors  $ y^{m} $, $ \; m=1,\ldots,n-1 $, we have 
\begin{align*}
  \psi_{m}&=\sqrt{n}\langle c,y^{m}\rangle =\sqrt{n}\langle c+k\mathbf{1} , y^{m}\rangle=\sqrt{n}\langle c^{\prime},y^{m}\rangle\\
  &= \psi^{\prime}_{m} \; \; \; \;   \forall m=1,\ldots,n-1\\
 \psi_{0}&=\langle c, \mathbf{1}\rangle,  \; \; \;  \psi_{0}^{\prime}=\langle c^{\prime},\mathbf{1}\rangle = \langle c, \mathbf{1} \rangle + \langle k \mathbf{1},\mathbf{1}\rangle\,.
\end{align*}
Furthermore $ \Cir(c) $ and $ \Cir(c^{\prime}) $ have different inverses. But by Lemma~\ref{T} $ T_{j}(c)=\sum_{m=1}^{n-1}w^{-jm}_{n}\psi^{-1}_{m} $. Since $ \psi_{m}=\psi_{m}^{\prime} $ we get $ T_{j}(c)=T_{j}(c^{\prime}),\; j=0,\ldots,n-1. $
Now let $ z\in \Z^{n} $ and $z^{\prime} = z + k\mathbf{1}$ be in the same layers as $ c $ and $ c^{\prime} $
respectively. With the help of Lemma~\ref{T0} we have
\begin{align}
n  \langle z,\overline{T}(c) \rangle &= z_0 T_0(c) + z_{1}T_{n-1}(c) + \dots + z_{n-1} T_{1}(c) + 1 \nonumber \\
                      &= z_0 T_0(c) + \dots + z_{n-1} T_{1}(c) + k (T_0(c) + \dots + T_{n-1}(c)) + 1 \nonumber \\
                      &= (z_0 + k) T_0(c) +(z_{1}+k) T_{n-1}(c)+ \dots + (z_{n-1} + k) T_{1}(c) + 1 \nonumber \\
                      &= z_0^\prime T_0(c) +z_{1}^{\prime}T_{n-1}(c)+ \dots + z^\prime_{n-1} T_{1}(c) + 1 =n  \langle z^\prime, \overline{T}(c^{\prime}) \rangle\,. \qedhere
\end{align}
\end{proof}

\begin{figure}
  \begin{center}
    \includegraphics[width=\textwidth]{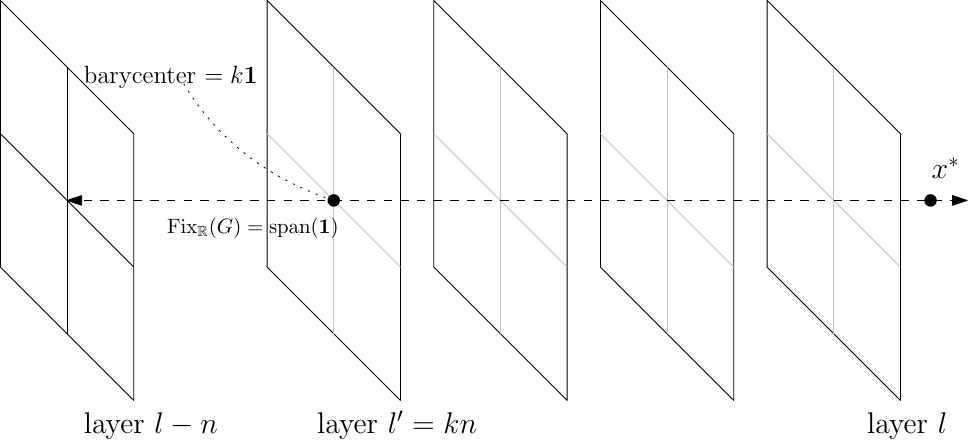}
  \end{center}
  \caption{Illustrating the proof of Lemma~\ref{n-layers}.}
  \label{fig:layers}
\end{figure}

The main idea of our algorithm for the transitive case (all variables active) is as follows.

\begin{lemma}
  \label{n-layers}
  Let $P$ be an integer linear maximization problem with $n$
  variables, objective function $f(x)=\langle \mathbf{1}, x \rangle$
  and transitive symmetry group $G$. Let the LP relaxation of $P$ have
  feasible region $R$ and optimal solution $x^*$.  Let $P_j$ be the
  integer feasibility problem obtained by intersecting $R$ with layer
  $j$.  Let $l=\lfloor f(x^*) \rfloor$. Define
  $F= \{ l-n < j \leq l \mid P_j \text{ is feasible} \}$.
  If $F$ is empty, then $P$ is infeasible; otherwise the optimal
  solution is in layer $k=\max(F)$.
\end{lemma}

\begin{proof}
  In the following discussion, we use \emph{above} and \emph{below} to
  refer to the natural ordering of layers by index.  Since we assume
  $G$ acts by permuting coordinates, there is no real loss of
  generality in fixing the objective function to $f(x)=\langle \mathbf{1}, x\rangle$.  As we
  already observed, since $G$ is transitive
  $\Fix_{\R}(G)=\Span(\mathbf{1})$.  It is clear that layer $l$ is the
  highest layer which could intersect $R$; furthermore if $F$ is non-empty
  the optimal solution of $P$ must lie in the highest layer with
  non-empty intersection with $R$. Suppose $F$ is empty, so no layer
  above $l-n$ is feasible.

  Define the \emph{orbit-barycenter} $b(x)$ for $ x\in \R^{n} $ as
  \begin{equation*}
    b(x)=\dfrac{1}{\vert G \vert}\sum_{g \in G} gx.
  \end{equation*}
  Note that this point lies in $\Fix_{\R}(G)$ and $\langle \mathbf{1}, x \rangle = \langle \mathbf{1}, b(x) \rangle$;
  furthermore if $x \in R$ then $b(x)\in R$. This allows us to assume
  w.l.o.g.\  that $x^*\in\Fix_{\R}(G)$.

  There exists integer $k$ satisfying $l-n < l'=kn \leq l$. The
  integer point $k\mathbf{1}$ is the unique intersection of layer $l'$
  and $\Fix_{\Z}(G)$, so since $P_{l'}$ is infeasible, $R$ does not intersect
  layer $l'$. Convexity then tells us that no layer below $l'$
  intersects $R$ (as the segment of $\Fix_{\R}(G)$ between such a layer
  and $x^*$ goes through $k\mathbf{1}$).
\end{proof}

Note that in order to check if an integer point $ c\in \Z^{n}_{(k)} $ is a core point or not it is impossible  to add  constraints of the form of~\eqref{lambda} for all  $ z\in  \Z^{n}_{(k)} $ because there are infinitely many integer points in each layer. But we can add constraints for a finite subset of integer points in $  \Z^{n}_{(k)}$.  As mentioned earlier, it makes sense to choose atom points and universal core points since they are the closest integer points to the barycenter. For this finite set we have the following definition.

\begin{definition}In each layer $ k $, any choice of  set of atoms and universal core points  for making inequalities~\eqref{lambda} is called an \emph{essential set}  of the layer $ k $ and is denoted by $ E^{k} $.
\end{definition}
Note that since vectors in $ \lbrace 0,1\rbrace^{n} $ are core points, a non-empty choice of essential set always exists.
\begin{definition}\label{cong}
We say that layers $ l $ and $ l^{\prime} $ in $ \Z^{n} $ are \emph{congruent} if $l\equiv l^{\prime}  $  $ \pmod {n} $.
\end{definition}

Lemma~\ref{T0}  plays an important role in defining the specific essential set for any layer.
\begin{remark}\label{100}
  We can translate each integer point $ z \in E^{k}$ through the fixed space to get an integer point with entries in $\{ -1,-2 ,0,1, 2\}$ and use
  inequalities~\eqref{lambda}, which can be written as
\begin{equation}\label{19}
1+ \langle \sigma^{-j}(z), (T_{0}(c),T_{n-1}(c),\ldots{},T_{1}(c)) \rangle <0.
\end{equation}
Since $ T_{0}(c)+T_{1}(c)+\ldots+T_{n-1}(c) =0$, the inequality~\eqref{19} holds for both or neither co-projective 
points $ z,z^{\prime} $ in congruent layers.
 \end{remark}
 
For example, for an ILP in $ \R^{6} $, we can define an essential set  in  layer $ 21 $ by:
\begin{equation*}
  E^{21}=\{ (4,4,4,3,3,3), (4,3,4,4,3,3),(4,3,4,3,4,3),(5,3,4,3,3,3)\}.
\end{equation*}
The corresponding constraint for $ z=(4,4,4,3,3,3) $, $j=0$ is
\begin{equation}\label{vay}
1+4T_{0}(c)+4T_{5}(c)+4T_{4}(c)+3T_{3}(c)+3T_{2}(c)+3T_{1}(c)<0. 
\end{equation}
Since by Lemma~\ref{T0} we have  $ T_{0}(c) +T_{1}(c)+T_{2}(c)+T_{3}(c)+T_{4}(c)+T_{5}(c)=0$, inequality~\eqref{vay} can be written as
\begin{equation*}
  1+T_{0}(c)+T_{5}(c)+T_{4}(c) <0.
\end{equation*}
Furthermore, we can define the essential set  as follows and use inequality~\eqref{19} for making new constraints.
\begin{equation*}
  E^{3}=\{ (1,1,1,0,0,0), (1,0,1,1,0,0),(1,0,1,0,1,0),(2,0,1,0,0,0)\}.  
\end{equation*}

 If we consider layers $1,\ldots{},n$ as representatives for all congruence classes, we  can define the essential set  for layers $ 1,\ldots{},n $ and use universal or atom points with entries $\{ -2,-1,0,1,2\}$.
  \begin{definition}
The essential set in any layers $ 1,\ldots{},n $ is called the \emph{projected essential set} and it is denoted  by $ \widehat{E}^{k} $ for $ k=1,\ldots{},n $.
\end{definition}

The idea is that  first we search for integer points whose circulant matrix is singular. Next we consider the case where the circulant matrix is non-singular: by adding new constraints in each layer $ k $ we search for an integer point $c$ whose orbit polytope does not contain the integer points of the essential set $ E^{k} $.  Since  $ \Cir(c) $ is non-singular in this step, we add the following constraints (cf.\ Theorem~\ref{7} and~\eqref{invert}, \eqref{inverT1}):
\begin{equation}\label{smooth}
\langle V_{m},c\rangle^{2}+\langle U_{m},c\rangle^{2} >0  \; \; \; \; \; \; \forall m=1,\ldots,\lceil (n-1)/2 \rceil.
\end{equation}
We will present variations on this idea for different kinds of symmetry groups. In all of these variations there are three different types of subproblems as follows:
\begin{enumerate}
\item[ ($ Q_{1}$)] Add constraints~\eqref{lambda} (or \eqref{6}) and \eqref{smooth} for each point in the projected essential set.
\item[ ($ Q_{2} $)] Add constraints~\eqref{in} and~\eqref{on} for the case of singular circulant matrices.
\item[ ($Q_{3}$)] Check the feasibility of integer points in  the essential sets. For Algorithm~\ref{alg:transitive} they can be tested individually, for Algorithms 2 and 3 see Remark~\ref{feas}.
\end{enumerate}

Algorithm~\vref{alg:transitive} follows the proof of Lemma~\ref{n-layers} for the case where
all variables are active.

\begin{algorithm}[hbt]
\caption{Maximize with all variables active}
\label{alg:transitive}
  \begin{algorithmic}[1]
  \Require Maximization ILP $P$ with bounded relaxation, transitive symmetry group, and objective function $f$.
  \Ensure Optimal objective value, or $-\infty$ if infeasible.
  \State Construct and solve a subproblem $P_0$ of type $Q_2$ by adding constraints~\eqref{in} and~\eqref{on} to $P$.
  \label{transitive-singular}
  \State If $P_0$ has integer optimal solution $ \widehat{z}_{1}$, set $f_1=f(\widehat{z}_1)$ otherwise set $f_1=-\infty $.
  \State Solve the LP  relaxation of $ P $, let $ x^{*} $ be the optimal solution and let $l=\lfloor f(x^{*})\rfloor$, $l'=\lfloor l/n\rfloor n$ (cf.~Figure~\ref{fig:layers}).
  \For{$i=l$ down to $l'+1$}
  \For{$j=1,\ldots{},m_{i}$}
  \State Construct and solve  subproblem $ P^{j}_{i} $ of type $Q_{3}$ for checking the feasibility of  $z^j\in E^{i}=\{z^{1},\ldots{},z^{m_{i}}\}$. \Comment{Unprojected essential set}.
  \If{$P^j_i$ has integer feasible solution $\widehat{z}_{2}$}
  \State \Return  $ \max \{f_1,f(\widehat{z}_{2})\} $
  \EndIf
  \EndFor
  \State Construct and solve a subproblem $ P_{i} $ of type $ Q_{1} $ by  adding constraints~\eqref{lambda}  for each $ z^{j} \in \widehat{E}^{i}$  along with constraints ~\eqref{smooth} for $0 \leq m \leq \lceil (n-1)/2 \rceil$.
  \label{transitive-nonsingular} \Comment{Projected essential set}
  \If{$P_i$ has integer feasible solution $\widehat{z}_{3}$}
  \State \Return  $ \max \{f_1,f(\widehat{z}_{3})\} $
  \EndIf
  \EndFor
  \State \Return $-\infty$
\end{algorithmic}
\end{algorithm}

Here are some observations related to Algorithm~\ref{alg:transitive}:
\begin{enumerate}
\item There might be some non-core integer points which satisfy  constraints~\eqref{lambda} for the given essential set.
\item By Lemma~\ref{T0}, in each layer $ L $  we need at most $\left\lceil (n-1)/2\right\rceil$ terms $ T_{k}, $ to make constraints~\eqref{lambda}. For example if $ n=6 $, in layer $ 5 $ for universal core point $ (1,1,1,1,1,0) $  we can use $1-T_{5} < 0  $ rather than $ 1+T_{0}+T_{1}+T_{2}+T_{3}+T_{4} <0. $
\end{enumerate}
The subproblems in steps~\ref{transitive-singular}
and~\ref{transitive-nonsingular} of Algorithm~\ref{alg:transitive} are NP-hard, so there are no known
worst case efficient algorithms. On the other hand, the cost of
generating constraints, solving the LPs, and testing points for
feasibility is all negligible. Actually, solving a symmetric LP can be
done faster since it is enough to search through the fixed space (cf.\
\cite[Theorem 1]{BHJ2013}). Since adding new constraints in
steps~\ref{transitive-singular} and~\ref{transitive-nonsingular} creates a nonlinear and non-convex mixed
integer program, we should use a solver that can handle these
constraints. In theory solving a nonlinear integer
program is harder than the linear one, but in practice we see that
(cf.\ Section~\ref{experiments}), decreasing the feasible region by
adding our constraints make it worthwhile to do so. It should also be
mentioned that there is a trade off between the number of integer
points in the essential set and running time of solving an ILP. Adding
too many constraints might make the problem hard to solve. Some
universal or atom points might be more effective for a specific
problem. However, in our experiments there is no significant
difference between different selected integer points in the essential
sets. We summarize this discussion as follows.
\begin{remark}\label{complex}
  \leavevmode
  \begin{enumerate}
  \item The main cost of Algorithm~\ref{alg:transitive} is in steps~\ref{transitive-singular} and~\ref{transitive-nonsingular} where nonlinear integer programs must be solved.
  \item The choice of essential set (as long as not too large) seems not to have a large impact.
  \end{enumerate}

\end{remark}

\subsection{Example: solving a symmetric ILP with Algorithm~1}
\label{sec:example}
In this section, we illustrate using Algorithm~1 to solve a small ILP.
Each of the subproblems discussed here is solved on commodity hardware in at most two seconds using Knitro~\cite{knitro}.  The integer linear program \emph{P1} with
cyclic symmetry group $C_{5}$ is defined~\cite{LS2018} as follows.
\begin{align*}
  \lequation{\text{minimize } x_0+x_1+x_2+x_3+x_4}\\*
  \shortintertext{ subj.\ to}
    515161x_0+18376x_1 -503804x_2 -329744x_3+  300011x_4 &\leq 59\\
    300011x_0 + 515161x_1+18376x_2 -503804x_3 -329744x_4 &\leq 59\\
    -329744x_0+300011x_1 + 515161x_2+18376x_3 -503804x_4& \leq 59\\
    -503804x_0 -329744x_1+  300011x_2+  515161x_3+18376x_4 &\leq 59\\
    18376x_0 -503804x_1 -329744x_2+  300011x_3 +515161x_4 &\leq 59\\
  \cequation{x_0+x_1+x_2+x_3+x_4 =1}\\
  \cequation{x_0, x_1,x_2,x_3,x_4 \in \Z}
\end{align*}

The group $C_{5} $ has a 1-dimensional fixed space generated by $ \mathbf{1} =(1,1,1,1,1) $ and two 2-dimensional real invariant subspaces $ \{ V_{1},U_{1} \} $ and $\{ V_{2},U_{2} \}$ where
\begin{align*}
V_{1}&=(1,\cos(\frac{-2\pi}{5}),\cos(\frac{-4\pi}{5}),\cos(\frac{-6\pi}{5}),\cos(\frac{-8\pi}{5}))\,,  \\
V_{2}&=(1,\cos(\frac{-4\pi}{5}),\cos(\frac{-8\pi}{5}),\cos(\frac{-12\pi}{5}),\cos(\frac{-16\pi}{5}))\,, \\
U_{1}&=(0, \sin(\frac{-2\pi}{5}),\sin(\frac{-4\pi}{5}),\sin(\frac{-6\pi}{5}),\sin(\frac{-8\pi}{5}))\,,\\
U_{2}&=(0, \sin(\frac{-4\pi}{5}),\sin(\frac{-8\pi}{5}),\sin(\frac{-12\pi}{5}),\sin(\frac{-16\pi}{5}))\,.
\end{align*}
Now if $ c=(c_{0},c_{1},c_{2},c_{3},c_{4}) $ is an integer solution of the above ILP  then the corresponding $ T_{i} $ in Theorem~\ref{7} are as below:
 \begin{align*}
T_{0}&=\dfrac{\langle V_{1},c \rangle}{\langle V_{1},c \rangle^{2}+\langle U_{1},c\rangle^{2}} +\dfrac{\langle V_{2},c \rangle}{\langle V_{2},c \rangle^{2}+\langle U_{2},c\rangle^{2}}\,, \\
 T_{1}&=\dfrac{\langle \sigma^{-1}(V_{1}),c \rangle}{\langle V_{1},c\rangle^{2}+\langle U_{1},c\rangle^{2}}+\dfrac{\langle \sigma^{-1}(V_{2}),c \rangle}{\langle V_{2},c \rangle^{2}+\langle U_{2},c\rangle^{2}}\,,  \\
 T_{2}&=\dfrac{\langle \sigma^{-2}(V_{1}),c \rangle}{\langle V_{1},c \rangle^{2}+\langle U_{1},c\rangle^{2}}+\dfrac{\langle \sigma^{-2}(V_{2}),c\rangle}{\langle V_{2},c \rangle^{2}+\langle U_{2},c\rangle^{2}}\,, \\
 T_{3}&=\dfrac{\langle \sigma^{-3}(V_{1}),c\rangle}{\langle V_{1},c \rangle^{2}+\langle U_{1},c\rangle^{2}}+\dfrac{\langle \sigma^{-3}(V_{2}),c \rangle}{\langle V_{2},c \rangle^{2}+\langle U_{2},c\rangle^{2}}\,, \\
 T_{4}&=\dfrac{\langle \sigma^{-4}(V_{1}),c \rangle}{\langle V_{1},c \rangle^{2}+\langle U_{1},c\rangle^{2}}+\dfrac{\langle \sigma^{-4}(V_{2}),c \rangle}{\langle V_{2},c \rangle^{2}+\langle U_{2},c\rangle^{2}}\,. 
\end{align*}
Since all variables are active we can use Algorithm~\ref{alg:transitive} to solve the problem. Note that since the problem is in layer 2, we do not need to solve an LP relaxation problem to determine layers. 

First we choose the essential set in layer 2 as  $E^{2} = \lbrace (1,1,0,0,0) \rbrace$. By adding the following nonlinear constraint,
\begin{equation*}
1+T_{0}(c)+T_{1}(c) <0,
\end{equation*}
we quickly (via Knitro) see the subproblem is infeasible.
Checking the feasibility of (1,1,0,0,0) is trivial.
By adding the following constraints (see~\eqref{in} and~\eqref{on}) we check if there is an integer solution that has a singular circulant matrix.
 \begin{align*}
 \langle V_{1},c\rangle^{2}+\langle U_{1},c\rangle^{2} &\leq  r_{1} (\langle V_{1},c\rangle^{2}+\langle U_{1},c\rangle^{2})\\
 \langle V_{2},c\rangle^{2}+\langle U_{2},c\rangle^{2} &\leq  r_{2} (\langle V_{2},c\rangle^{2}+\langle U_{2},c\rangle^{2})\\
\cequation{-2r_{3}\leq c_{0}+c_{1}+c_{2}+c_{3}+c_{4}\leq 2r_{3}}\\
\cequation{r_{1}+r_{2}+r_{3} \leq 2}
\end{align*}
The final subproblem was also solved quickly by Knitro, and no integer solution was found.

\section{Partial-Circulant Matrices}
In the previous section we assumed that the order of the cyclic permutation group of an ILP was equal to the dimension of the problem. In this section, we generalize the algorithm of the previous section for some ILP where not all variables are active.

For the remainder of this section, we denote by $x(1)$ the coordinates
of $x$ that are active in the cyclic group $C_{k}$.  We extend this notation below
to direct products of
several cyclic groups. Without loss of generality, assume $x(1)$ is
the first $k<n$ coordinates of $x$. Notice that in this case the group
action is not transitive. Indeed, the dimension of the fixed space is
$ n-k+1 $ and it is spanned by the orthogonal vectors
$e_{0} +\cdots+ e_{k-1}, e_{k},\ldots, e_{n-1}$ (using the usual basis
vectors).  Moreover, invariant subspaces of the action of $ C_{k} $ on
$ \R^{k} $ embed naturally into invariant subspaces for the action of
$ C_{k} $ on $ \R^{n} $. The following definition is a generalization
of the circulant matrix that will be useful in generalizing the
constraints of the previous section.
\begin{definition}\label{1}
For $ c\in \R^{n} $, an $ n \times k $ \emph{Partial-Circulant Matrix}, $\PCir(c)$, is a matrix where the first k rows are $\Cir(c(1))$ and the remaining rows are $\mathbf{1}^{T}$ scaled by the last $n-k$ elements of $  c$.
\end{definition}
\begin{example} Let $ c=(1,2,3,4,5) \in \Z^{5} $. Then the $ 5\times 3  $ partial circulant matrix of $ c $ takes the following form
  \begin{equation*}
  \begin{bmatrix}
1 & 3& 2\\
2 & 1& 3 \\
3 &2 &1 \\
4&4 &4 \\
5& 5& 5
\end{bmatrix}.
  \end{equation*}
\end{example}
Let $ c\in \Z^{n} $ and suppose the circulant part of $ \PCir(c) $ is invertible. Then the orbit polytope of $ c $ has dimension $ k-1 $ and any point $ x \in \R^{n} $ in $ \conv(G_{c}) $ can be written as
\begin{equation}\label{101}
  \PCir(c)_{n\times k} \lambda= x \; \; \; \text{for some}\; \; \lambda\in \R^{k}. 
\end{equation}
The rank of this partial-circulant matrix  is $ k $ and so the solution $ \lambda$ of the system~\eqref{101} can be determined by the circulant part.

Suppose the permutation $ g\in S_{n} $ factors as the product $g = h_{1} \cdot h_2 \cdots h_{d}$ of disjoint cycles $h_{j}$. Let $X_{1},\ldots, X_{d}$ be the canonically associated subspaces of $ \R^{n} $. Thus the cyclic group $H_{j}=\langle h_{j}\rangle$ permutes standard basis vectors for $ X_{j} $. Furthermore, each integer point $z \in \Z^{n}$ has a unique decomposition $  z= \oplus_{j=1}^{d}z(j)$ with $z(j) \in X_{j} \cap \Z^{n}$.
Keeping this notation, we have that the next lemma follows from the definition of  a convex combination. 

\begin{lemma}\label{4.1}
Suppose $ G=\langle g \rangle \leq S_{n}$, where $g = h_{1}\cdot h_2\cdots h_{d}$, is a product of disjoint cycles. Let $ c\in \Z^{n} $. If the integer point $  z= \oplus_{j=1}^{d}z(j)$ is in $\conv(G_{c})$, then for $j \in 1,\ldots, d$, we have $z(j) $ in the orbit polytope of $ c(j) $ under $ H_{j} $.
\end{lemma}

In this section we are only need the restricted case of $d=1$ of Lemma~\ref{4.1}: if the orbit polytope of $ c(1)=(c_{0},\ldots,c_{k-1}) $ in $\Z^{k}$ is lattice-free then $ c=(c_{0},\ldots,c_{n-1}) $ is a core point in $ \Z^{n} $. Moreover, since the sum of weights is 1 in a convex combination, if $ z\in \Z^{n} $ is in $ \conv(G_{c}) $ then $z_{k+1}=c_{k+1}$, \ldots{}, $z_{n}=c_{n}$.
Furthermore, applying inequality~\eqref{lambda},   if $ z$ does not lie in  $ \conv(G_{c}) $, then for at least  one cyclic permutation of $ z(1) $ the following constraint must be satisfied
\begin{equation}\label{h}
z_{0}T_{0}(c)+z_{1}T_{k-1}(c)+\ldots+z_{k-1}T_{1}(c) +1 < 0.
\end{equation}
Recall that the above constraint is nonlinear and non-convex with respect to $c$.

For a given symmetric (maximization) ILP where not all variables are
active in the cycle, let the optimal solution of its relaxation be
$ x^{*} $. In this case the symmetry group does not act
transitively. So, searching in $ k $ layers is not sufficient and the
optimal objective value can be in any layer less than or equal to
$ \lfloor f(x^{*})\rfloor$. Nonetheless we can define a finite set of
subproblems where we can apply the constraints developed in
Section~\ref{sec:circulant}. In the following discussion, we use
\emph{sub-layer $j$} to refer to the set of integer points whose first
$k$ coordinates sum to~$j$.

\begin{lemma}\label{AA}
  Given an ILP $P$ with a cyclic symmetry group
  $G=\langle (1,2,\ldots{},k) \rangle \leq S_{n}$ acting on the first
  $k$ coordinates,
  \begin{enumerate}
  \item  $P$ can be solved by solving $k$ subproblems created by adding the constraint
    \begin{equation}
      \label{eq:sub-layer}
      \sum_{j=0}^{k-1}x_j = i \pmod k  \quad \text{for }i\in\{1\ldots k\}\,.
    \end{equation}
  \item We can use the constraints~\eqref{lambda} in each subproblem.
  \item Each of the subproblems can be modelled by adding a single
    integer variable and a single integer linear constraint to $P$.
  \end{enumerate}
\end{lemma} 
\begin{proof}
  Let $ L=l+l^{*} $ be the layer of an integer point $z \in \Z^{n}$ in
  $P$ where $l$ is the layer of active variables and $l^{*}$ is the
  layer of non-active variables.  The point $z$ must fall into one of
  the $k$ congruence classes defined by~\eqref{eq:sub-layer}, which
  establishes the first point.

  By Lemma~\ref{iso} all congruent sub-layers $l$ have the same
  constraints for co-projective integer points (see
  Definition~\ref{equiv} and Definition~\ref{cong}). Indeed congruence
  is not needed here, only that co-projective points form equivalence
  classes for~\eqref{lambda}.

  Although the constraint~\eqref{eq:sub-layer} is nonlinear, it is
  easy to model using integer-linear constraints.  Consider the
  congruence relation between layers (see Definition~\ref{cong}). Let
  $ i=1,\ldots,k $ be the representative of each
  class.   Let $q_{i}$, be an integer solver variable.
 Constraint~\eqref{eq:sub-layer} can be reformulated as
  \begin{equation*}
    \sum_{j=0}^{k-1} x_j =  q_{i}k+i\,.
    \qedhere
  \end{equation*}
\end{proof}

Algorithm~\ref{alg:transitive}  in the previous section can be modified  in this case as Algorithm~\ref{alg:first-k} below. The difference is we should check sub-layers $l_{i}= q_{i}k+i$ where  $ q_{1},\ldots{},q_{k} $ are some new integer variables (rather than a fixed layer).
For example, if $k=3$, then all layers in $ \Z^{3} $ can be classified as
\begin{equation*}
  l_{1}= 3q_{1}+1 , \; l_{2}=3q_{2}+2 , \text{ or }  l_{3}=3q_{3}+3 .
\end{equation*}
Note that since the algorithm for this case is not searching  layer by layer, checking the feasibility of integer points in the essential set is different. In other words, we know the layers of points
we have (since  by Remark~\ref{100} the integer points in each projected essential set  are in the layers   $ 1,\ldots{},k $) but they are only representatives for (many) pre-images, and we don't know the layers of the pre-images.  If   $ z(1)=\sigma^{i}(c(1)) $ then the solution of $ \Cir(c(1))\lambda=z(1) $ is $ e_{a} $, that is:
\begin{equation*}
  \lambda_{a}=1 , \; \; \; \; \lambda_{i}=0 \; \; \; \forall i=0,\ldots{},k-1 ,\; i\neq a,
\end{equation*}
where $ a\in \{ 0,\ldots,k-1\} $. Furthermore, by Theorem~\ref{7}, to check if at least one of the  pre-images of $ z(1)$ in the projected essential set is a feasible point or not, for $ a=0 $, we can add the following constraints:
\begin{align*}
\langle \sigma^{0}(z(1)), \overline{T}(c(1)) \rangle &=1,\\
\langle \sigma^{1}(z(1)), \overline{T}(c(1)) \rangle &=0,\\
 \vdots&\\
\langle \sigma^{k-1}(z(1)), \overline{T}(c(1)) \rangle &=0.
\end{align*} 
\begin{remark}\label{feas}
Since the coordinates of integer points in the projected essential set are 0,1,-1,2,-2, another way to check the feasibility of all translates along the fixed space (which uses only linear constraints with small coefficients) is to chose one of the coordinates $ c_{i} $  as a base and write other coordinates with respect to that coordinate. 
\end{remark}

For example if $ z= (1,-2,0,0,0,0,0) \in \widehat{E}^{1}$, then by choosing $ c_{2} $ as a base we can add the following constraints to check the feasibility of all integer points $(1,-2,0,0,0,0,0)+t(1,1,1,1,1,1,1),   t \in \Z, $ in the sub-layers $ 7t+1$ :
\begin{align*}
  c_0 &= c_{2}+1, \\
  c_1 &= c_{2}-2, \\
  c_3 &= c_{2}, \\
  c_4 &= c_{2}, \\
  c_5 &= c_{2}, \\
  c_6 &= c_{2}\,. \\
\end{align*}

Algorithm~\vref{alg:first-k} generalizes
Algorithm~\ref{alg:transitive} to the case where $k<n$ coordinates are
active in the symmetry group.  For conciseness, we write $\OPT(P)$ for
the optimal objective value of maximization problem $P$. If $P$ is
infeasible, $\OPT(P)$ is defined as $-\infty$. In practice this can be
implemented by a Boolean flag, and does not require any specialized
arithmetic.

\begin{algorithm}[htb]
  \caption{Maximize with first $k$ variables active.}
  \label{alg:first-k}
\begin{algorithmic}[1]
  \Require Maximization ILP $P$ with bounded relaxation, cyclic symmetry group acting on the first $k$ coordinates, and objective function $f$.
  \Ensure Optimum objective value, or $-\infty$ if $P$ is infeasible.
  \State Construct and solve a subproblem $P_0$ of type $Q_2$ by adding constraints~\eqref{in} and~\eqref{on} to $P$.
  Set $f^*$ to $\OPT(P_0)$.
  \label{first-k:singular}
 \For{$i= 1 \ldots k$}
 \State Construct a subproblem $L_i$ by adding  constraint $\sum_{j=0}^{k-1}x_{j}=q_{i}k+i$ where $q_i$ is a new integer  variable.
 \State Choose projected essential set $\widehat{E}^{i}=\{z^{1},\ldots{},z^{m_{i}}\}$.
 \For{$ j=1,\ldots{},m_{i}$}
 \State Construct and solve subproblem $P^{j}_{i} $  of type $ Q_{3} $ by adding  constraints described in Remark~\ref{feas}  to $L_i$ for point in $z^j  \in \widehat{E}^{i}$
 \label{first-k:feas}
 \State $f^* \leftarrow \max \{f^*, \OPT(P^j_i)\}$.
 \EndFor
 \State Construct a subproblem $ P_{i} $ of type $ Q_{1} $ by adding  to $L_i$ constraint~\eqref{smooth} for all  $ m=0,\ldots,\left\lceil (n-1)/{2}\right\rceil $ and constraints~\eqref{lambda} for each $ z^{j} \in \widehat{E}^{i} $.
 \label{first-k:nonsingular}
 \State  $ f^{*} \leftarrow \max \{ f^*, \OPT(P_i)\}.$
 \EndFor
\State \Return $f^*$
\end{algorithmic}
\end{algorithm}

As discussed in Remark~\ref{complex}, giving precise complexity bounds
for Algorithm~\ref{alg:first-k} is difficult. But it is worthwhile to
compare the complexity of Algorithm~\ref{alg:transitive} and
Algorithm~\ref{alg:first-k}. In step 1 of of both algorithms the same
constraints are used but the difference is that, in the second
algorithm there are some non-active coordinates. So the subproblem in
this case is harder to solve since constraints are effective on a
small dimension of the problem. The type of constraints in
step~\ref{first-k:nonsingular} of Algorithm~\ref{alg:first-k} and
step~\ref{transitive-nonsingular} of Algorithm~\ref{alg:transitive} is
the same but there are two differences: 1) less coordinates are again
active in the second algorithm, and 2) in the first algorithm in each
subproblem is in a single layer, but in the second algorithm in each
subproblem is solved over an equivalence class of sub-layers. Both of
these factors make the type $Q_1$ subproblems harder to solve in
Algorithm~\ref{alg:first-k}. The other important fact is that, in the
first algorithm since we are searching layer by layer, the first
integer solution (if there is any) is an optimal solution. But in the
second algorithm since all coordinates are not active if an integer
solution is found it might not be optimal.  On the other hand, when
using Algorithm~\ref{alg:first-k} for feasibility problems the
subproblems can also be solved as feasibility problems, using local
solvers if desired.  Finally, checking the feasibility of universal
and atom core points in these two algorithms is done differently. In
the first algorithm in each subproblem, it is just testing an integer
point against a set of linear inequalities. But in the second
algorithm since not all coordinates are active, in each subproblem an
ILP should be solved after adding constraints described in
Remark~\ref{feas}.

\section{New Constraints for Direct Products of Cyclic Groups}\label{EXT1}

Lemma~\ref{AA} shows that  Algorithm~\ref{alg:first-k} can be used for a cyclic subgroup of the symmetric group of an ILP. In some cases the cyclic subgroup is small and Algorithms 1 and 2 are not very practical. In this section we generalize Algorithm~\ref{alg:first-k} for direct products of cyclic groups. Like Algorithm~\ref{alg:first-k}, the algorithm in this section also does not search layer by layer in the feasible region. Instead, we search in all equivalence classes of sub-layers.

Recall that the Cartesian product of the $ d $ sets $X_{1},\ldots, X_{d}$ is
\begin{align*}
X_{1}\times \ldots{}\times X_{d}= \prod_{i=1}^{d} X_{i}=\{ (x_{1},\ldots{},x_{d}): \; x_{i}\in X_{i} \; \text{for every } i \in \{1,\ldots{},d\} \} .
\end{align*}
\begin{definition}
Let $ G_{i} , i=1,\ldots{},d$ be some finite groups. The \emph{direct product} $ G_{1}\times \ldots{} \times G_{d} $
is defined as follows.
\begin{enumerate}
\item The underlying set is $ G_{1} \times \ldots{} \times G_{d} $. 
\item Multiplication is defined coordinate-wise:
\begin{equation*}
(g_{1},\ldots{},g_{d})\cdot (g_{1}^{\prime},\ldots{},g_{d}^{\prime}) =  (g_{1}.g^{\prime}_{1},\ldots{},g_{d}.g_{d}^{\prime}).
\end{equation*}
\item The identity element of this group is defined as $(e_{1},\ldots,e_{d})$ where $ e_{j} $ is the identity element of $ G_{j} $ for $ j=1,\cdots ,d. $
\end{enumerate}
\end{definition}
It is  routine to check that with the above operation, $\prod_{i=1}^{d} G_{i} $ is a group.
  Consider a product $g = h_{1}\cdot h_2\cdots h_{d}$ of disjoint cycles in $ S_{n} $, and let $H_{i} = \langle h_{i}\rangle $.
\begin{lemma}\label{direct}
The direct product $G = H_{1}\times \cdots\times H_{d}$ is isomorphic to $G^{\prime} = \langle h_{1},\ldots,h_{d}\rangle$.
\end{lemma}
\begin{proof}
Let $g_i = h_{i}^{t_i}$  denote a general element of the cyclic group $H_{i}$.  Consider the following map between $ G $ and $ G^{\prime} $ 
\begin{align*}
  \phi: G &\rightarrow G^{\prime}\\
(g_{1},\ldots{},g_{d}) &\rightarrow g_{1}g_{2} \cdots g_{d},
\end{align*}
It is straightforward to check $ \phi $ is a bijection since the cycles are disjoint ($g_{i}$'s commute since the $h_{i}$'s are disjoint). Also we have
\begin{align*}
  \phi((g_{1},\ldots{},g_{d})\cdot (g_{1}^{\prime},\ldots{},g_{d}^{\prime}))=\phi(g_{1} g^{\prime}_{1},\ldots{},g_{d} g_{d}^{\prime})= g_{1} g^{\prime}_{1} \ldots{} g_{d}g_{d}^{\prime}= \\
  g_{1} \ldots{} g_{d} g_{1}^{\prime} \ldots{} g_{d}^{\prime} =  \phi(g_{1},\ldots{},g_{d})\phi(g_{1}^{\prime},\ldots{},g_{d}^{\prime}).
\end{align*}
So $ \phi $ is a group isomorphism. Note that the second to last equality holds since the cycles are disjoint.
\end{proof}
\begin{example}
The permutation group $ G=\langle  (1,2,3,4),(5,6,7) \rangle $ on the set $\{1,2,3,4,5,6,7\}$ is isomorphic to the direct  product of two cyclic groups $ H_{1}=\langle (1,2,3,4)\rangle $ and $ H_{2}=\langle (5,6,7)\rangle $
\end{example}
In this section we denote by $ z(i) $ the coordinates of $ z $ that are active  in the  i-th cyclic group in the direct product $  H_{1} \times \ldots{} \times H_{d}$.

The following theorem states that if a permutation group $ G\leq S_{n} $ is a direct product of other permutation groups $ G^{i} $, (that is, permutation groups $G^{i}$ acting on disjoint subsets of coordinates in the usual way), then the core set of $  G$ is also a Cartesian product. 
\begin{theorem}\emph{\cite[Theorem 8]{5}}\label{cart}
  Let $ G= \prod_{i=1}^{d} G^i $, $ G^i\leq S_{n}$. Then
  \begin{equation*}
    \core(G) =\prod_{i=1}^{d}\core(G^i).
  \end{equation*}
\end{theorem}
\begin{proof}
The product structure of $  G$ induces a decomposition of $ \R^{n} $ into a Cartesian product of pairwise orthogonal coordinate subspaces $\oplus_{i=1}^{d} X_{i}= \R^{n}$. Thus, we can write every point $ z\in \R^{n} $ as $z =\oplus_{i=1}^{d} z(i)$. The claim of the theorem follows immediately from $ \conv(G_{z})= \prod_{i=1}^{d} \conv(G^i_{z(i)}) $.
\end{proof}
We are concerned with various subgroups of the direct product  $\prod_{j=1}^{d}H_{j} $, where the cyclic groups $ H_{j}=\langle h_{j}\rangle $ are generated by disjoint cycles $ h_{1},\ldots,h_{d} $ in $ S_{n} $. Let $ k_{j} $ denote the length (period) of
the cycle $ h_{j} $. Thus, if  $ g=h_{1}\cdots h_{d} $ is the complete factorization of the permutation $ g \in S_{n} $ into disjoint (possibly trivial) cycles, then $ k_{1}+ \cdots k_{d}=n $. 

Note that we can use the same $\widehat{T}(c) $ as in Lemma~\ref{T} for each cycle of the direct product groups. In other words, let  $ c =\bigoplus_{i=1}^{d} c(i)\in \Z^{n} $. When required, we will suppose that the vector $ c $ has the property that the matrices $\Cir(c(j))$ are invertible, $1\leq j\leq d$.  Lemma~\ref{T}  then tells us the corresponding inverse matrices  are  $\Cir( \widehat{T}(c(j)))$.
 Then  a point $ x =\bigoplus_{i=1}^{d} x(i)\in \R^{n}$ is in $ \conv(G_{c}) $   if and only if
\begin{align*}
 x(1)=\Cir(c(1))\lambda^{1} &\Rightarrow \lambda^{1}=\Cir(\widehat{T}(c(1))) x(1)\,,
\\
 &\vdots \\
 x(d)=\Cir(c(d))\lambda^{d} &\Rightarrow \lambda^{d}=\Cir(\widehat{T}(c(d))) x(d)\,.
\end{align*}
Furthermore,  by Theorem~\ref{cart} we can apply our new constraints on each subspace. 

 We  denote by $ l^{j}_{i}$ the sub-layer of  active variables of an integer point in the $ j$-th cyclic group and  in the  $ i$-th equivalence class, and  denote by  $ \widehat{E}^j_i$ the  corresponding projected essential set of each sub-layer $ l_{i}^{j}. $

Similarly to the previous section, all equivalence classes of all sub-layers must be checked.   By Lemma~\ref{iso}, Remark~\ref{100} and Lemma~\ref{direct}, in each cyclic group $H_{i} $ we need to search in $ k_{i}$ sub-layers. Then there are $ k_{1}\times \ldots{} \times k_{d}$ possibilities for sub-layers of an integer point. 
For example suppose $ G= \langle  (1,2,3)\rangle \times \langle(4,5)\rangle \leq S_{5}$. Then we have decomposition $ \Z^{n}=X_{1} \times X_{2} $ where $ z(1)=(z_{0},z_{1},z_{2}) $  and $ z(2)=(z_{3},z_{4}) $.  So we have the following possibilities for sub-layers  of a feasible solution $ z $
\begin{align*}
  l^{1}_{1}&=3q_{1}+1 , \; \; l^{2}_{1}=2q_{2}+1,\\
  l^{1}_{2}&=3q_{1}+2 , \; \; l^{2}_{1}=2q_{2}+1,\\
  l^{1}_{3}&=3q_{1}+3 , \; \; l^{2}_{1}=2q_{2}+1,\\
  l^{1}_{1}&=3q_{1}+1 , \; \; l^{2}_{2}=2q_{2}+2,\\
  l^{1}_{2}&=3q_{1}+2 , \; \; l^{2}_{2}=2q_{2}+2,\\
  l^{1}_{3}&=3q_{1}+3 , \; \; l^{2}_{2}=2q_{2}+2,
\end{align*}
where $ q_{1} $ and $ q_{2} $ are arbitrary integers.

Consider a cycle $ h_{j} $. Since its length is $ k_{j} $, a complete set of residues modulo $ k_{j} $ is the set $ [k_{j}]:=\{1,\ldots,k_{j}\} $. We require their Cartesian product 
\begin{equation*}
 K=[k_{1}]\times \ldots{}\times [k_{d}]=\{(t_{1},\ldots{},t_{d}): \; t_{j} \in [k_{j}] \; 1\leq j \leq d \}.  
\end{equation*}
For each element $ (t_{1},\ldots{},t_{d}) $ of $K \subset \Z^n$, a subproblem of type $ Q_{1} $ is solved, where $ t_{j},\;  j=1,\ldots{},d$, is the sub-layer of  active variables of cycle $ h_{j}. $
\begin{remark}\label{comb}
  In the direct product group $G= \prod_{i=1}^{d} H_{i}$, each
  sub-layer has its own projected essential set. In step~\ref{first-k:feas} of
  Algorithm~\ref{alg:first-k} the feasibility of pre-images of integer points of the
  projected essential sets is checked. In the direct product case we can check the
  feasibility of these pre-images corresponding to each cycle
  separately. If
  \begin{math}
    s_j := \sum_{i=1}^{k_j} |\widehat{E}^j_i|, 
  \end{math}
  then $\sum_{j=1}^ds_j$ subproblems must be solved.

  On the other hand if adding constraints for one cycle is not enough,
  we can check for simultaneous feasibility.
  Let $(t_1, \dots, t_d) \in K$ be a vector of residues (where $K$ is
  defined above).  Given points $z(j) \in \widehat{E}^j_{t_j}$, $j=1\ldots d$,
  we can check simultaneous feasibility of $ z(1),\ldots,z(d)$. In
  this case, $\Pi_{j=1}^d s_{j}$ subproblems must be solved.
\end{remark}

For  example if $ G=C_{5} \times C_{8} $ and a universal core point in the projected essential set $ \widehat{E}^{1}_{2} $ is $ (1,1,0,0,0) $, by Remark~\ref{feas}, we add the following constraints to the problem: 
\begin{align*}
 x_{1}&=x_{0}\,,\\
 x_{2}&=x_{0}-1\,, \\
 x_{3}&=x_{0}-1\,, \\
 x_{4}&=x_{0}-1\,.
\end{align*} 

We can also check all possible combinations of integer points in the
projected essential sets. For example if
$ G= \langle (1,2,3,4)\rangle \times \langle (5,6,7)\rangle \leq S_{7}$ ,
for $ l_{2}^{1} $ and $ l_{1}^{2} $ let
$ \widehat{E}_{2}^{1} =\{ (1,1,0,0),(1,0,1,0)\}$ and
$ \widehat{E}_{1}^{2} =\{ (1,0,0), (2,-1,0)\}$. Then all possible
combinations of $ \widehat{E}_{2}^{1} $ and $\widehat{E}_{1}^{2}$ are
as below:
\begin{equation*}
 (1,1,0,0,1,0,0), (1,1,0,0,2,-1,0),(1,0,1,0,1,0,0),(1,0,1,0,2,-1,0) .
\end{equation*}

Now we can generalize Algorithm~\ref{alg:first-k} for direct product groups as shown
in Algorithm~\vref{alg:product}. For simplicity we present the
version which tests the essential points individually for feasibility; the
modification of the loop at step~\ref{loop:feas} for simultaneous
testing is straightforward.

\begin{algorithm}[htb]
  \caption{Maximization for ILPs with direct product cyclic symmetry}
  \label{alg:product}
\begin{algorithmic}[1]
  \Require{maximization ILP $ P $ with symmetry group $G=\Pi_{j=1}^dH_j$, bounded relaxation, objective function $f$.}
  \Ensure{The optimal objective value, or $-\infty$ if $P$ is infeasible.}
  \State  $f^* \leftarrow -\infty$
  \State Let $k_j$ be the order of cyclic group $H_j$.
  \For{$ t = (t_{1},\ldots{},t_{d})\in [k_{1}]\times \ldots{} \times [k_{d}] $}
  \State Choose projected essential sets  $\widehat{E}_t = \bigcup \{\, \widehat{E}^1_{t_1},\ldots{},\widehat{E}^d_{t_d} \,\}$.
  \For{$z \in \widehat{E}_t$} \label{loop:feas}
  \State Construct and solve subproblem $ P^{z}_{t} $ of type  $Q_{3}$  by adding constraints described in Remark~\ref{feas} to $ P $ for point $z$.\label{prod:q3}
  \State $f^*\leftarrow \max \{f^*, \OPT(P^z_t) \}$
  \EndFor
  \Statex\quad Construct and solve a subproblem $P_{t}$  by adding to~$P$:
  \For{$a=1,\ldots{},d,$}
  \If {$t_{a}=k_{a}$}
  \State Constraints~\eqref{in} and~\eqref{on}. \Comment{$Q_{2}$}.
  \Else 
  \State Constraints~\eqref{smooth} for $0 \leq m \leq \left\lceil (n-1)/2\right\rceil$ and
  constraints~\eqref{lambda} for each $z$ in $\widehat{E}_t$. \Comment{$Q_{1}$}.
    \EndIf
    \EndFor
    \State $f^{*} \leftarrow \max\{ f^*, \OPT(P_t)\}$
    \EndFor
\State \Return $f^*$
\end{algorithmic}
\end{algorithm}

\begin{remark}
  Compared to Algorithm~\ref{alg:first-k}, both the number of
  subproblems and their difficulty are potential bottlenecks in
  Algorithm~\ref{alg:product}. While the steps are similar to the
  Algorithm~\ref{alg:first-k}, the running time of each subproblem is
  longer. In particular in each subproblem of step~\ref{prod:q3} we
  either have constraints acting on a small set of
  coordinates (if we check each subgroup individually), or a very
  large number of subproblems (if we check for simultaneous
  feasibility).
\end{remark}
\section{New Constraints For Permutation Groups}
In this section we use Algorithms 3, 2 or 1   for a symmetric ILP with some permutation group as its symmetry group. First we define subdirect product groups and then classify the permutation groups with respect to their  generators.

\begin{definition}
Let $ H_{1} $ and $ H_{2} $ be groups. A \emph{subdirect product} of $ H_{1} $ and $ H_{2} $ is a subgroup $ H $ of the external direct product $H_{1} \times H_{2}$ such that the projection from $H$ to either direct factor is surjective. In other words, if $p_{1}:H_{1} \times H_{2} \to H_{1} $ is given by $(h_{1},h_{2}) \mapsto h_{1} $ and $p_{2}:H_{1} \times H_{2} \to H_{2} $ is given by $(h_{1},h_{2}) \mapsto h_{2}$, then $ p_{1}(H) = H_{1} $ and $p_{2}(H) = H_{2}$.
Subdirect products of more than two groups are obtained naturally by iterating this construction.
\end{definition}
\begin{example}
The group $\langle  (1,2)(3,4,5,6) \rangle $ is  a subdirect product of $ \langle (1,2) \rangle $ and $ \langle (3,4,5,6) \rangle $ and is a subgroup of the direct product $\langle (1,2) \rangle \times \langle (3,4,5,6) \rangle$.
\end{example}

\begin{remark}\label{sub}
Suppose $ g\in G \leq S_{n} $ is factored as a product of disjoint cycles. These cycles themselves may or may not be elements of $  G$.
But  Lemma~\ref{4.1}  shows that for exploring feasible integer points of a symmetric polyhedron considering some of the cycles is enough. In other words, we can make our constraints for a few cycles and keep other variables non-active. 
\end{remark}
Note that our constraints are still valid for a subdirect product of disjoint cycles, since in each sub-problem we remove universal or atom points (by generating constraints from their projections) from the orbit polytope of active variables of each cycle. So Algorithms 2 and 3 work for this case as well.

Now we  classify permutation groups with respect to their generators.

\paragraph{Disjoint Cycles}

If $ G $ is a group generated by disjoint cycles $ h_{1},\ldots,h_{d} $, then $ G $ is the direct product of $ H_{1},..,H_{d} $ (Lemma~\ref{direct}). In this case  Algorithm~\ref{alg:product} can be applied.
\begin{example}
The group $ G=\langle (1,2,3,4),(6,8,9),(5,7,12,11) \rangle $ is the direct product of the three cyclic groups $ H_{1}=\langle (1,2,3,4) \rangle $,  $ H_{2}=\langle (6,8,9) \rangle $, $ H_{3}=\langle (5,7,12,11) \rangle. $
\end{example}

\paragraph{Product of disjoint cycles}

If the given set of generators of a group $G$ is a single permutation which is the product of two or more disjoint cycles, then $ G $ is a subdirect product of the corresponding cyclic groups.  In this case by Remark~\ref{sub}, Algorithm~\ref{alg:product} can be applied.
\begin{example}
The subdirect product group $ G=\langle  (1,2,5,3)(6,10,11)(9,7,12,8)  \rangle $ is a subgroup of the direct product of three cyclic groups $ H_{1}=\langle (1,2,5,3) \rangle $,  $ H_{2}=\langle (6,10,11) \rangle $, $ H_{3}=\langle (9,7,12,8) \rangle. $
\end{example}

\paragraph{Combination of two of the above cases}
If the generators of $G$ are products of various numbers of disjoint cycles, then Algorithm~\ref{alg:product} can be applied.
\begin{example}
Let $ G=\langle(1,2,3,4)(5,6,12,13), (7,8,9,10)\rangle$, then $ G $ is a subgroup of the direct product of three cycles $ H_{1}=\langle (1,2,3,4)\rangle $, $ H_{2}=\langle (5,6,12,13)\rangle $, $H_{3}=\langle (7,8,9,10)\rangle $. 
\end{example}

\paragraph{Non-disjoint cycles}
If  the given set of generators of $ G $ are not disjoint cycles, we find a subgroup of $ G $ where all generators are disjoint.  Then it falls into one of the previous three cases. Note that for finding the biggest cyclic subgroup we can find representatives of the conjugacy classes (e.g.\ by using GAP~\cite{GAP4}) and choose the biggest cycle.
\begin{example}
Let $ G= \langle (1,2,3,4,5),(7,5,8,10,11)(12,13,2,14)   \rangle $ then $ H_{1}=\langle (1,14,12,13,2,3,4,8,10,11,7,5)  \rangle  $ is a cyclic subgroup and since 9 and 6 are fixed, Algorithm~\ref{alg:first-k} can be applied. 
\end{example}

\section{Computational Experiments}

\label{experiments}

In order to test the efficiency of our algorithms, in this section we create some symmetric integer linear programs that are hard to solve with standard solvers. For this purpose, we can use the orbit polytopes of core points. As the orbit polytope of a core point contains no integer points aside from the vertices, if we can cut off these vertices then the integer program corresponding to this polytope will be infeasible. 

Infeasible problems are typically hard for branch-and-bound algorithms
because there is no chance of early success.  The goal of the
\emph{integer feasibility} problems is to find an integer point in a
polyhedron $P$ or decide that no such point exists.  Aaarden and
Lenstra presented some difficult integer feasibility
problems~\cite{AL04,GZ02} whose relaxations are simplices. Our
techniques are not suitable for the examples of~\cite[Table 1]{AL04}
because these simplices lack symmetries in the sense of
Remark~\ref{2}. Symmetric instances are interesting in their own right
because equivalent partial solutions can blow up the
branch-and-bound-tree (cf.~\cite{PR2019}). In order to construct
symmetric examples that are still integer infeasible, we start with
lattice-free symmetric simplices and cut off the (integer) vertices; a
similar technique is used in~\cite{2}.

The convex hull of the orbit of a core point under
a cyclic group with an invertible $ n\times n $ circulant matrix is a
simplex with dimension $ n-1$.
We say an integer point has a \emph{globally minimal projection} (with respect to some invariant subspace $V$) if
\begin{equation*}
  \Vert z\vert_{V} \Vert \leq \Vert z^{\prime}\vert_{V} \Vert \text{ for all } z^{\prime} \in \mathrm{aff}( G_{z} ) \bigcap \Z^{n}.
\end{equation*}
For a primitive group the following
theorem shows that the corresponding orbit polytope is often a
simplex.
\begin{theorem}\emph{ \cite[Theorem 5.37]{2}}
  \label{thm:5.37}
Let $ G\leq S_{n} $ be primitive and let $ V\leq \R^{n} $ be a rational invariant subspace. If $ e_{0}=(1,0,0,\ldots{},0)  $ has globally minimal projection onto $ V $,  then there are infinitely many core points in layer one. The corresponding orbit polytopes are simplices.
\end{theorem}

Based on Theorem~\ref{thm:5.37} the core point used to define the
instances is calculated as $p+\sigma w$ for some universal core point
$p$ and direction $w$ orthogonal to the fixed space.  We refer to the
scalar $\sigma$ as the \emph{scale factor} or just \emph{scale}.  The
groups used for these experiments are the primitive groups with GAP
identifiers $(15, 2)$, $(21, 2)$, $(45, 1)$ (all three of them have
rational invariant subspaces).  Table~\ref{Table} gives more details
about symmetry groups of these instances.

Preliminary experiments indicated that for these problems, there was
little benefit to using Algorithm 3 instead of Algorithm 2.  In
Section~\ref{sec:example} we presented an example where the
subproblems were solved using Knitro~\cite{knitro}. While Knitro
performs well in our setting, we wanted to ensure our results were
replicable with open source software. For this reason we have
implemented Algorithm 2 using
PySCIPOpt~\cite{PySCIPOpt}.  In order to simplify the nonlinear
subproblems $Q_1$, we have only implemented the case of single element
essential sets, using \eqref{6}. 

We use Algorithm 2, CPLEX, Gurobi and SCIP~\cite{SCIP} to solve P2,
P3, and P4 and compare the results.  We also checked that as expected
the analogous feasible instances (simplices) are generally easy for
all solvers, and that Algorithm~\ref{alg:first-k} finds the expected
solution.  All experiments in this section were carried out a Dell
PowerEdge R640 server with a Xeon Gold 6248 (2.5GHz) processor and
384GB of RAM. The operating system was Debian Linux 12.10 (amd64).
CPLEX was version 22.1.1 and Gurobi was version 12.0.1. We used SCIP
9.2.1 and PySCIPOpt 5.3.0. All results are the mean of 3 runs. We use
a 1 hour time limit for all runs.

\begin{table}[tbp]
 \caption{ ILP-feasibility problems.}
 \label{Table}
 \begin{tabularx}{\textwidth}{@{} c>{\hsize=0.5\hsize}C
>{\hsize=1.5\hsize}C
                              @{} }
 \toprule
     Name  & Symmetry group & Generators of the group \\
 \hline
P2 &  Primitive group (15,2) & $(1,15,7,5,12) (2,9,13,14,8) (3,6,10,11,4)$,  $ (1,4,5)(2,8,10)(3,12,15)(6,13,11)(7,9,14)$ \\
 \hline
P3 &   Primitive group (21,2) & $(1,7,12,16,19,21,6)(2,8,13,17,20,5,11)\cdot$ $(3,9,14,18,4,10,15)$,
 $(4,6,5)(9,11,10)(13,15,14)(16,18,17)(19,20,21)$  \\
 \hline

   P4 & Primitive group (45,1) & $( 1, 2, 7)( 3,11,27)( 4,14,31)( 5,18,32)( 6,20,36)\cdot$
                                 $( 8,24,39)(9,25,28)(10,26,42)(12,15,16)(13,30,40)\cdot$
                                 $(17,19,21)(22,35,44)(23,33,29)(34,43,37)\cdot$
                                 $(38,45,41)$,
                                 $(1,3,5,6,7,22,13,23)(2, 8, 9,10)\cdot$
                                 $(4,15,16,17,14,21,19,12)\cdot$
                                 $(11,28,29,38,44,25,20,37)\cdot$
                                 $(18,33,34,24,40,36,41,39)\cdot$
                                 $(26,43,35,32,42,45,27,30)$,
                                 $( 1, 4)( 3,12)( 5,19)( 6,21)( 7,14)( 8,10)(11,20)\cdot$
                                 $(13,16)(15,23)(17,22)(18,33)(24,41)(25,28)\cdot$
                                 $(26,43)(27,32)(29,44)(30,35)(34,39)(36,40)(42,45) $ \\
 \bottomrule
 \end{tabularx}
\end{table}

\begin{figure}[tbp]
\centering
\caption{Comparison of elapsed time for problem P2. Asterisks indicate timeout.}
\label{fig:P2}
\scriptsize
\begin{tikzpicture}
\begin{axis}[
    xlabel={Scale},
    ylabel={Elapsed Time (s)},
    xmin=1, xmax=20,
    ymin=0, ymax=3800,
    xtick={1,2,...,20},
    ytick={0,500,1000,2000,3000,3600},
    legend pos=north west,
    grid=major,
    thick,
    scatter/classes={a={mark=o},b={mark=asterisk}},
    width=0.85\textwidth,
    height=0.55\textwidth,
]

\addplot[color=blue, mark=square*]
coordinates {
(1,2.03)(2,4.08)(3,3.85)(4,5.60)(5,4.83)(6,4.01)(7,22.38)(8,10.87)(9,34.07)(10,33.71)
(11,46.66)(12,103.73)(13,48.11)(14,127.67)(15,66.69)(16,689.64)(17,1885.64)(18,1275.66)(19,725.15)(20,3475.37)
};
\addlegendentry{Algorithm 2}

\addplot[scatter,color=red,scatter src=explicit symbolic]
coordinates {
  (1,0.05) [a]
  (2,0.14) [a]
  (3,1.67) [a]
  (4,3.66) [a]
  (5,7.33) [a]
  (6,11.39) [a]
  (7,24.20) [a]
  (8,17.25) [a]
  (9,78.15) [a]
  (10,299.95) [a]
  (11,468.63) [a]
  (12,126.68) [a]
  (13,1166.74) [a]
  (14,1146.67) [a]
  (15,1097.42) [a]
  (16,1886.85) [a]
  (17,2610.94) [a]
  (18,3600) [b]
  (19,3600) [b]
  (20,3600) [b]
};
\addlegendentry{CPLEX}

\addplot[color=green!60!black, mark=triangle*]
coordinates {
(1,0.15)(2,0.17)(3,0.56)(4,3.65)(5,9.98)(6,16.30)(7,1.31)(8,34.60)(9,80.68)(10,72.56)
(11,35.55)(12,0.52)(13,310.69)(14,360.81)(15,298.17)(16,646.88)(17,776.52)(18,775.60)(19,1631.70)(20,2205.33)
};
\addlegendentry{Gurobi}

\addplot[color=purple, mark=diamond*]
coordinates {
(1,0.29)(2,0.64)(3,0.66)(4,1.61)(5,1.41)(6,7.07)(7,2.65)(8,7.56)(9,11.65)(10,13.19)
(11,47.36)(12,73.86)(13,16.13)(14,50.03)(15,129.70)(16,163.80)(17,160.27)(18,187.94)(19,264.65)(20,233.04)
};
\addlegendentry{SCIP}

\end{axis}
\end{tikzpicture}
\end{figure}

\begin{table}[tbp]
\centering
\caption{Solution times (seconds) and memory usage (kilobytes) for P2.}
\label{table:comparison-infeasible-problems-11-20}
\captionsetup{font=scriptsize} 
\scriptsize
\begin{tabular}{|c|r|r|r|r|r|}
\hline
\textbf{Solver} & \textbf{Scale} & \textbf{Elapsed Time (s)} & \textbf{Kernel (s)} & \textbf{User (s)} & \textbf{Memory (kB)} \\ \hline
alg2 & 11 & 46.66 & 0.26 & 45.56 & 93492\\ \hline
cplex & 11 & 468.63 & 0.01 & 468.56 & 25168\\ \hline
gurobi & 11 & 35.55 & 0.01 & 35.50 & 15728\\ \hline
scip & 11 & 47.36 & 0.12 & 47.23 & 44688\\ \hline
alg2 & 12 & 103.73 & 0.29 & 102.45 & 110489\\ \hline
cplex & 12 & 126.68 & 0.00 & 126.62 & 21583\\ \hline
gurobi & 12 & 0.52 & 0.00 & 0.48 & 15083\\ \hline
scip & 12 & 73.86 & 0.23 & 73.40 & 85620\\ \hline
alg2 & 13 & 48.11 & 0.28 & 46.93 & 99559\\ \hline
cplex & 13 & 1166.74 & 0.03 & 1166.61 & 47459\\ \hline
gurobi & 13 & 310.69 & 0.00 & 310.64 & 17308\\ \hline
scip & 13 & 16.13 & 0.14 & 16.00 & 49881\\ \hline
alg2 & 14 & 127.67 & 0.37 & 126.24 & 117356\\ \hline
cplex & 14 & 1146.67 & 0.03 & 1146.55 & 42931\\ \hline
gurobi & 14 & 360.81 & 0.01 & 360.75 & 19704\\ \hline
scip & 14 & 50.03 & 0.22 & 49.68 & 74865\\ \hline
alg2 & 15 & 66.69 & 0.31 & 65.50 & 104793\\ \hline
cplex & 15 & 1097.42 & 0.01 & 1097.34 & 21613\\ \hline
gurobi & 15 & 298.17 & 0.01 & 298.12 & 17077\\ \hline
scip & 15 & 129.70 & 0.40 & 128.77 & 155315\\ \hline
alg2 & 16 & 689.64 & 1.03 & 686.77 & 292491\\ \hline
cplex & 16 & 1886.85 & 0.05 & 1925.08 & 54880\\ \hline
gurobi & 16 & 646.88 & 0.02 & 646.81 & 21843\\ \hline
scip & 16 & 163.80 & 0.12 & 2.76 & 49144\\ \hline
alg2 & 17 & 1885.64 & 2.52 & 1879.69 & 743260\\ \hline
cplex & 17 & 2610.94 & 0.13 & 2610.65 & 104028\\ \hline
gurobi & 17 & 776.52 & 0.02 & 776.43 & 19444\\ \hline
scip & 17 & 160.27 & 0.51 & 159.07 & 128180\\ \hline
alg2 & 18 & 1275.66 & 1.34 & 1271.31 & 458885\\ \hline
gurobi & 18 & 775.60 & 0.01 & 775.51 & 20956\\ \hline
scip & 18 & 187.94 & 0.54 & 186.53 & 111533\\ \hline
alg2 & 19 & 725.15 & 1.18 & 722.11 & 446423\\ \hline
gurobi & 19 & 1631.70 & 0.01 & 1631.59 & 22445\\ \hline
scip & 19 & 264.65 & 0.64 & 262.91 & 179401\\ \hline
alg2 & 20 & 3475.37 & 4.22 & 3465.94 & 1312812\\ \hline
gurobi & 20 & 2205.33 & 0.05 & 2205.17 & 37880\\ \hline
scip & 20 & 233.04 & 0.77 & 231.14 & 197280\\ \hline
\end{tabular}
\end{table}

We first discuss problem P2.  As shown in Figure~\ref{fig:P2}, for
scales 1 to 10, all solvers complete in less than 5 minutes.
Table~\ref{table:comparison-infeasible-problems-11-20} shows a
comparison of time and memory use for scales 11 through 20. The
built-in symmetry handling of SCIP (not used by Algorithm 2) performs
well on these problems.  For scales 11-20 Algorithm 2 is generally
faster than CPLEX and Gurobi, and in particular completes all 20
scales within the 1 hour time limit, while CPLEX times out for scales 18,
19, and 20.

\begin{table}[tbp]
\centering
\caption{\normalsize Solution times and memory usage for P3.}
\label{tab:P3}
\scriptsize
\begin{tabular}{|c|r|r|r|r|r|}
\hline
\textbf{Solver} & \textbf{Scale} & \textbf{Elapsed Time (s)} & \textbf{Kernel (s)} & \textbf{User (s)} & \textbf{Memory (kB)} \\ \hline
alg2   & 1 & 9.24     & 0.17  & 9.17     & 74957  \\ \hline
cplex  & 1 & 0.78     & 0.003 & 0.77     & 16981  \\ \hline
scip   & 1 & 2.00     & 0.11  & 1.99     & 28257  \\ \hline
gurobi & 1 & 5.96     & 0.007 & 5.95     & 15997  \\ \hline
alg2   & 2 & 149.42   & 0.29  & 149.22   & 117337 \\ \hline
cplex  & 2 & 20.95    & 0.000 & 20.95    & 17921  \\ \hline
scip   & 2 & 5.63     & 0.11  & 5.63     & 29745  \\ \hline
gurobi & 2 & 91.47    & 0.003 & 91.46    & 16329  \\ \hline
alg2   & 3 & 363.08   & 0.377 & 362.77   & 164063 \\ \hline
cplex  & 3 & 108.43   & 0.010 & 108.41   & 21687  \\ \hline
scip   & 3 & 191.34   & 0.31  & 191.13   & 78953  \\ \hline
gurobi & 3 & 1075.75  & 0.017 & 1075.72  & 23979  \\ \hline
alg2   & 4 & 1074.83  & 0.63  & 1074.26  & 219821 \\ \hline
cplex  & 4 & TO       & TO    & TO       & -      \\ \hline
scip   & 4 & TO       & TO    & TO       & -      \\ \hline
gurobi & 4 & 2342.67  & 0.033 & 2342.59  & 43643  \\ \hline
alg2   & 5 & 2036.46  & 1.357 & 2035.11  & 369835 \\ \hline
cplex  & 5 & TO       & TO    & TO        & -      \\ \hline
scip   & 5 & TO        & TO     & TO        & -      \\ \hline
gurobi & 5 & 4350.08  & 0.06  & 4349.00  & 75039  \\ \hline
\end{tabular}

\end{table}

Table~\ref{tab:P3} presents the performance of Algorithm 2, CPLEX,
SCIP, and Gurobi on problem P3 for scales
 1 to 5. The solvers did not solve the problem for scales
6 to 10 in less than 1 hour.  Algorithm 2 and Gurobi are the only
solvers that complete for 5 scales --- two more
than CPLEX and SCIP, which solve instances only up to scale 3 before
exceeding the 1 hour time limit.

\begin{table}[tbp]
\centering
\caption{Solution times and memory usage for P4.}
\label{tab:P4}
\scriptsize
\begin{tabular}{|c|r|r|r|r|r|}
\hline
\textbf{Solver} & \textbf{Scale } & \textbf{Elapsed Time (s)} & \textbf{Kernel (s)} & \textbf{User (s)} & \textbf{Memory (kB)} \\ \hline
alg2   & 1 & 81.54   & 0.31  & 81.23   & 94292  \\ \hline
cplex  & 1 & 2.97    & 0.003 & 2.94    & 18261  \\ \hline
scip   & 1 & 7.25    & 0.11  & 7.21    & 41339  \\ \hline
gurobi & 1 & 6.64    & 0.00  & 6.61    & 16517  \\ \hline
alg2   & 2 & 1081.36 & 0.64  & 1080.66 & 209033 \\ \hline
cplex  & 2 & 103.02  & 0.007 & 103.01  & 24896  \\ \hline
scip   & 2 & 65.00   & 0.29  & 64.81   & 173423 \\ \hline
gurobi & 2 & 751.74  & 0.02  & 751.71  & 19220  \\ \hline
alg2   & 3 & TO      & TO    & TO      & -      \\ \hline
cplex  & 3 & 1813.18 & 0.04  & 1813.06 & 44367  \\ \hline
scip   & 3 & TO      & TO    & TO      & -      \\ \hline
gurobi & 3 & TO      & TO    & TO      & -      \\ \hline
\end{tabular}

\end{table}

Table~\ref{tab:P4} compares the performance of Algorithm 2, CPLEX,
SCIP, and Gurobi on the infeasible ILP problem P4 for scales 1 to
3. All solvers complete for scales 1 and 2. At scale 3, only CPLEX
completes, while the others time out. For scales 4 to 10, all solvers
time out.  Among the four, Gurobi demonstrates the best performance at
smaller scales in terms of both runtime and memory usage. SCIP also
shows strong performance on scale 2 but with higher memory
usage. These instances seem challenging for all tested solvers.

\section{Conclusions}

In this paper we introduced some new techniques for solving symmetric
linear programs based on deriving nonlinear constraints from the
symmetry group of the formulation. Since these constraints depend only
on the symmetry group, the same constraints can be re-used for many
problems. The practical benefits of using a nonlinear solver in order
have a smaller search space need more evaluation, but at least on the
artificial instances in Section~\ref{experiments} they show some
promise, allowing the solution of instances not solvable within 1
hour on the same hardware using commercial MILP solvers. We think our
methods may be useful for problems with a large enough cyclic subgroup
in their symmetry group, particularly those where determining integer
feasibility is a challenge.  From a theoretical point of view, these
techniques show that core point techniques are not limited to groups
where the number of (non-equivalent) core-points is finite.

\section*{Acknowledgements}
The authors wish to thank two anonymous referees for many helpful
suggestions, including the simplified proof of Lemma~\ref{T0} that is
given here. They would also like to thank Barry Monson, Branimir Ćaćic
and Nicholas Touikan for helpful feedback on earlier versions of the
paper.

\bibliographystyle{spmpsci}
\bibliography{references}

\end{document}